\documentclass[11pt,a4paper]{article}
\usepackage[T1]{fontenc}
\usepackage{authblk}
\usepackage[colorlinks,linktocpage,linkcolor=blue]{hyperref}
\usepackage{amsfonts}
\usepackage{amssymb}
\usepackage{graphicx}
\usepackage{amsfonts}
\usepackage{fancyhdr}
\usepackage{float}
\usepackage{subfigure}
\usepackage{epstopdf}
\usepackage{algorithm}
\usepackage{amsthm, amsmath, enumitem}
\usepackage{color}
\usepackage{multirow}
\usepackage{algorithm}
\usepackage{algpseudocode}
\usepackage{amsmath}
\usepackage{bm}
\usepackage[section]{placeins}

\newtheorem{theorem}{Theorem}[section]
\newtheorem{lemma}{Lemma}[section]

\newtheorem{remark}{Remark}[section]

\newtheorem{example}{Example}

\def\bq{\begin{equation}}
\def\eq{\end{equation}}
\def\br{\begin{eqnarray}}
\def\er{\end{eqnarray}}
\def\brr{\bq\begin{array}{rlll}}
\def\err{\end{array}\eq}

\def\text#1{\hbox{#1}}

\newcommand{\bsub}{\begin{subequations}}
\newcommand{\esub}{\end{subequations}$\!$}

\newcommand{\n}{\partial\vec n}
\graphicspath{{Figure/}} 

\usepackage{bm}
\usepackage{mathrsfs}%
\usepackage{cite}

\numberwithin{equation}{section}
\begin{document}
\title{Inverse source problem for the parabolic equation with sparse moving observations}
\author[1]{Qiling Gu\thanks{guql@sustech.edu.cn}}
\author[1]{Wenlong Zhang\thanks{zhangwl@sustech.edu.cn}}
\author[2,3]{Zhidong Zhang\thanks{zhangzhidong@mail.sysu.edu.cn}}
\affil[1]{\normalsize{Department of Mathematics, Southern University of Science and Technology (SUSTech), Shenzhen, Guangdong, China.}}
\affil[2]{\normalsize{School of Mathematics (Zhuhai), Sun Yat-sen University, Zhuhai 519082, Guangdong, China}}
\affil[3]{\normalsize{Guangdong Province Key Laboratory of Computational Science, Sun Yat-sen
University, Guangzhou 510000, Guangdong, China}}
\maketitle

\begin{abstract}
This paper considers the inverse problem of identifying the source term of parabolic equations from sparse boundary measurements. We used data from moving sensors to locate the unknown source term. This work first proves the uniqueness of the inverse problem under such measurements. Then the movement strategy of the sensor is given, from which the authors build the reconstruction algorithm. Finally, some numerical experiments are performed and the corresponding results are generated, which indicate the effectiveness of the algorithms.    
\end{abstract}

\date{}


\hskip\parindent

\section{Introduction.}\label{secInt}
\subsection{Mathematical statement.}

The model considered in this work is given as follows:
\begin{equation}\label{PDE}
\begin{cases}
\begin{aligned}
(\partial_t-\Delta)u(x,t)&=\chi_{{}_{x\in D}}, &&(x,t)\in\Omega\times(0,T],\\
u(x,t)&=0, &&(x,t)\in\partial\Omega\times(0,T]\cup\Omega\times\{0\},
\end{aligned}
\end{cases}
\end{equation}
where the domain $\Omega$ is set as the unit disc in $\mathbb R^2$. For equation \eqref{PDE}, the source term is the characteristic function $\chi_{{}_{D}}$ with unknown support $D$. The inverse problem is to recover the support $D$ with the measurements of the boundary flux data, which could be given as 
\begin{equation}\label{data}
\frac{\partial u}{\n}(z(t),t),\ z(t)\in \partial\Omega,\ t\in(0,T].
\end{equation}
For data \eqref{data}, $\vec n$ is the outward normal vector of the boundary $\partial \Omega$; the notation $z(t)$ indicates that the observation point is moving. Such data \eqref{data} could be obtained from moving sensors. In summary, the inverse problem investigated in this article could be described as 
\begin{equation}\label{inverse problem}
\text{to recover the unknown support}\ D\ \text{with data \eqref{data}.}
\end{equation}

\subsection{Background and literature.}

Inverse source problems for parabolic equations have significant applications in medical imaging~\cite{RundellZhang:2020}, environmental monitoring~\cite{LinOuZhangZhang:2024}, geophysical exploration~\cite{HelinLassasYlinenZhang:2020} and so on. To save observational and computational costs, solving inverse problems with sparse measurements has drawn more and more attention. For instance, \cite{ChengLiu:2020} considers the recovery of a source term in a parabolic equation from local measurements taken at two distinct time instants. In \cite{HettlichRundell:2001}, iterative regularization schemes are developed for identifying a discontinuous source from flux data measured over time at two separate boundary points. The inverse source problem for an advection-dispersion-reaction system is studied in \cite{ElBadiaHamdi2007}, where a time-dependent point source is recovered from concentration and flux measurements at only two boundary points. In a related work, the authors investigates the identification of a point source in a linear advection-dispersion-reaction equation using concentration data collected at two strategically located observation points, from which both its location and time-dependent intensity are reconstructed \cite{ElBadiaHaDuongHamdi2005}. For fractional subdiffusion equations, \cite{Hrizi2024SinglePoint} proposes a second-order non-iterative algorithm based on the topological derivative to reconstruct a singular time-dependent source from a single-point measurement. In \cite{Li2023Inverse}, the fractional derivative order is identified by solving a nonlinear equation derived from a single space-time point measurement, exploiting the properties of the Mittag-Leffler function. Sparse-data reconstruction has also been investigated in related wave propagation problems, including inverse source reconstruction with multi-frequency sparse data and inverse scattering with phaseless or incomplete measurements; see, e.g., \cite{AlzaaligHuLiuSun2020Sparse,JiLiu2021SparseScattered,LiuMeng2023SparseMeasurements,LiLiu2023SparseObs,ZhangZhang2017PhaselessMF,ZhangZhang2018FastImaging,ZhangZhang2020ApproxFactorization}.
 In \cite{RundellZhang:2020}, the authors prove that for the heat equation on the unit disc in $\mathbb R^2$, a variable separable source $F(x,t)=p(x)q(t)$ with unknown $p$ and $q$ can be uniquely determined from flux measurements at two chosen boundary points. This work is extended to the cases of fractional diffusion equations \cite{LiZhang:2020}, and the parabolic equation on a general domain \cite{LinZhangZhang:2022}.

Restricting the observation region to a small size could  obviously decrease the cost for solving inverse problems. However, it would cause a natural question: how to determine the location of observation region, or where to put the sensors. This question is crucial since \cite{LinOuZhangZhang:2024} has the evidence that the location of observation region substantially affects the quality of numerical results. To address the influence of sensors' locations on numerical results, in \cite{LinOuZhangZhang:2024} the authors decide to use the moving observation points. More precisely, the sensors' locations would be functions of time. Under this setting, people could move the sensor to a 'better' place if they could determine the current location is 'bad'. Moreover, this setting is common to some extent in the practical applications. For example, in atmospheric science, the sensors could be set in a car or boat and work while the vehicle is moving.

\subsection{Contributions and Outline.}

This work could be regarded as the promotion of \cite{LinOuZhangZhang:2024}. In \cite{LinOuZhangZhang:2024} the authors also use the boundary flux data generated from two moving sensors to recover the discontinuous source, with one assumption that the path $z_j(t)$ of each sensor is a piecewise constant function. This assumption is obviously infeasible in practical applications. For this article, we summarize the contributions and novelty as follows. 

\begin{enumerate}
    \item \textbf{The piecewise constant assumption on the path of sensors is not required anymore.} In this work, we prove uniqueness theorem of inverse problem \eqref{inverse problem} under a continuous path of observation point. This setting is much more reasonable in practical applications.  
    \item \textbf{We used only 1 observation point.} We could ensure the uniqueness theorem with data from one moving sensor; while in \cite{LinOuZhangZhang:2024}, the uniqueness theorem is valid under two boundary observation points. 
    Theoretically, the costs of one or two observation points may differ slightly, since they are both subsets with zero measures of the boundary $\partial \Omega$. However, in engineering, the cost of each sensor could not be neglected. Hence, compared with \cite{LinOuZhangZhang:2024}, this work could save the cost of inverse problem further. 
    \item \textbf{The movement strategy of sensor is given.}
    With the idea of moving sensors, people could move observation points to better places if the current locations are bad. However, how to determine the location is good or not, and how to reasonable move the sensor, are not answered in \cite{LinOuZhangZhang:2024}. This work gives an explicit movement strategy of sensor, stated in \eqref{movement}. Then the corresponding algorithms are built. The performance of numerical results could confirm the effectiveness and rationality of the movement strategy.    
\end{enumerate}

The remainder of this paper is organized as follows. Section \ref{sec_uniqueness} concerns with the theoretical part of this article, proving the uniqueness theorem of inverse problem \eqref{inverse problem}. Section \ref{sec_numerical} is devoted to the numerical reconstruction and its experimental validation, including the discretization of the forward problem, the framework of Bayesian inverse problem and the movement strategy of sensor. In Section \ref{sec_experiment}, several numerical experiments are taken and the corresponding results are presented.

\section{Uniqueness of inverse problem \eqref{inverse problem}.}\label{sec_uniqueness}

\subsection{The Laplacian eigensystem \texorpdfstring{$\{\lambda_n,\varphi_n(x)\}$}{eigenvalues and eigenfunctions}.}

The eigensystem $\{(\lambda_n,\varphi_n)\}_{n=1}^\infty$
(multiplicity counted) of the operator
$-\Delta$ on the unit disc $\Omega$ with Dirichlet boundary condition is defined as follows:
$$
0<\lambda_1\le \cdots \le \lambda_n\le \cdots
\to\infty,\quad \mbox{as $n\to \infty$,}
$$
and $\varphi_n$ denotes the corresponding eigenfunction
\begin{equation}\label{eigenfunction}
\varphi_n(r,\theta)=\omega_nJ_{|m(n)|}(\lambda_n^{1/2}r)e^{im(n)\theta },
\ n\in\mathbb N^+,
\end{equation}
which form an orthonormal basis of $L^2(\Omega)$.
Here $(r,\theta)$ are the polar coordinates on $\Omega$, and
$J_{|m(n)|}(\cdot)$ is the Bessel function of order $|m(n)|$
with $\lambda_n^{1/2}$ as its zero point. The Bessel orders $m$
depend on the choice of $n$ and we use the notation $m(n)$ to show
the dependence (sometimes we may use $J_m$ for short).
\begin{remark}
$\{\omega_n\}$ are the normalized coefficients to make sure
$\|\varphi_n\|_{L^2(\Omega)}=1$.
From Bourget's hypothesis, proved in \cite{Siegel:2014},
there exist no common positive zeros between two Bessel functions
with different nonnegative integer orders. Also recall that
$J_{-m}(r)=(-1)^mJ_m(r)$, given an
eigenvalue $\lambda_{n_0}$, $\lambda_{n_0}^{1/2}$ can only be the zero
of $J_{\pm m(n_0)}(\cdot)$. Hence the multiplicity for $\lambda_{n_0}$
is two if $m(n_0)$ is nonzero, otherwise, it will be one.

In the case of $m(n_0)\ne 0$, by setting
$\lambda_{n_0}=\lambda_{n_0+1}$, the corresponding eigenpairs are given as
$$(\lambda_{n_0}, \omega_{n_0}J_{|m(n_0)|}(\lambda_{n_0}^{1/2}r)e^{i|m(n_0)|\theta }),
\quad (\lambda_{n_0+1}, \omega_{n_0+1}J_{|m(n_0)|}(\lambda_{n_0+1}^{1/2}r)e^{-i|m(n_0)|\theta }).$$
Now setting $m(n_0)=|m(n_0)|=-m(n_0+1)$, the representation \eqref{eigenfunction} is consistency and $m$ is uniquely determined by the value of $n$. See \cite{GrebenkovNguyen:2013} for details about the structure of $\{\varphi_n\}_{n=1}^\infty$.
\end{remark}

\subsection{Forward operator.}
From \cite{LinZhangZhang:2022}, we could give the following results for the forward operator $F$, which maps the unknown support $D$ to the boundary data. 

\begin{lemma}\label{lemma_forward_operator}
We set $d_n=\int_D \overline{\varphi_n(x)}\ dx$ as the Fourier coefficients of $\chi_{{}_{D}}$. 
Then we have 
\begin{equation}\label{forward}
\begin{aligned}
F(\{d_n\}_{n=1}^\infty)&=\frac{\partial u}{\n}(z(t),t)\\
&=-\sum_{n=1}^\infty C_{z(t),n}d_n (1-e^{-\lambda_n t}), 
\end{aligned}
\end{equation}
where 
$$z(t)=(\cos {\theta_{z(t)}},\sin{ \theta_{z(t)}})\in\partial\Omega,\quad C_{z(t),n}=\pi^{-1/2}\lambda_n^{-1/2}e^{-im(n)\theta_{z(t)}}.$$
\end{lemma}

\subsection{Uniqueness theorem.}
In our prior work \cite{RundellZhang:2018,RundellZhang:2020}, we comprehensively analyzed the uniqueness of this inverse problem, which requires the data from two chosen observation points. In this work, due to the setting of moving sensor, we could decrease the cost from two points to one moving observation point with specific path, which is stated as follows.

\begin{theorem}\label{T1}
Given two unknown supports $D_1$ and $D_2$ which have sufficiently smooth boundaries, we denote the solutions corresponding to $D_j$ by $u_j$, $j=1,2$. For data \eqref{data}, we assume that there exists $0\le t_1<\tilde t_1<t_2<\tilde t_2\le T$ such that
\begin{equation*}
z(t) = (\cos {\theta_1},\sin{\theta_1})\ \text{on}\ (t_1, \tilde t_1),\ z(t) = (\cos {\theta_2},\sin{\theta_2})\ \text{on}\ (t_2, \tilde t_2),\ \text{and}\ (\theta_1 - \theta_2)/\pi \notin \mathbb{Q},
\end{equation*}
where $\mathbb{Q}$ is the set of rational numbers. If
\[ \frac{\partial u_1}{\partial \mathbf{n}}(z(t),t) = \frac{\partial u_2}{\partial \mathbf{n}}(z(t),t), \ t \in (0,T), \]
then $D_1 = D_2$ in the sense of $L^2(\Omega)$.
\end{theorem}

\begin{proof}
With \cite[Lemma 3.5]{RundellZhang:2020}, we could uniquely extend the data $\frac{\partial u_1}{\n}$ from $(t_1, \tilde t_1)$ and $(t_2, \tilde t_2)$ to $(0,\infty)$, respectively.  
Then with the data 
$$\frac{\partial u}{\n}\Big|_{(\cos {\theta_1},\sin{\theta_1})\times(0,\infty)},\quad \frac{\partial u}{\n}\Big|_{(\cos {\theta_2},\sin{\theta_2})\times(0,\infty)}$$ and the condition $(\theta_1 - \theta_2)/\pi \notin \mathbb{Q}$, we could deduce the uniqueness result following \cite[Theorem $1$]{RundellZhang:2020}. The proof is complete. 
\end{proof}

\section{Numerical reconstructions.}\label{sec_numerical}

Although Theorem~\ref{T1} ensures uniqueness for boundary flux data collected along a moving sensor path, the inverse problem remains severely ill-posed in practice. Hence, reconstruction quality depends strongly on the informativeness of the measurements \cite{LinZhangZhang:2022, RundellZhang:2020}. Since a static sensor configuration may be ineffective when only a few sensors are available, we adopt a moving sensing strategy that steers a mobile sensor along the boundary toward more informative regions.

\subsection{Discretization for the forward problem.}

We consider the heat equation on the unit disk
\[
\Omega := \{(x,y)\in\mathbb{R}^2 \mid x^2+y^2<1\},
\]
where the source term is given by the indicator function $\chi_{{}_{x\in D}}$ of an unknown subregion $D\subset\Omega$, which is the target of the reconstruction. The governing equation reads
\begin{equation}\label{eq:main_pde}
\frac{\partial u}{\partial t} - \Delta u = \chi_{{}_{x\in D}}, \qquad x\in\Omega,\ t\in(0,T],
\end{equation}
subject to the homogeneous Dirichlet boundary condition and the zero initial condition
\begin{align*}
u(x,t) &= 0, \qquad x\in\partial\Omega,\ t\in[0,T],\\
u(x,0) &= 0, \qquad x\in\Omega.
\end{align*}
The inverse problem is to reconstruct the unknown domain $D$ from measurements of the boundary heat flux $\partial u/\partial \nu$ collected at selected locations on $\partial\Omega$ over time.

To exploit the circular geometry of $\Omega$, we rewrite the problem in polar coordinates $(r,\theta)$, where $r=\sqrt{x^2+y^2}$ and $\tan\theta=y/x$. The spatial domain then becomes
\[
\tilde{\Omega}:=\{(r,\theta)\mid 0<r<1,\ 0\le \theta\le 2\pi\},
\]
and \eqref{eq:main_pde} takes the form
\begin{equation}\label{eq:polar_pde}
\frac{\partial u}{\partial t}
-\left(
\frac{1}{r}\frac{\partial}{\partial r}\left(r\frac{\partial u}{\partial r}\right)
+\frac{1}{r^2}\frac{\partial^2 u}{\partial \theta^2}
\right)
= \chi_{{}_{x\in D}}(r,\theta),
\qquad (r,\theta)\in\tilde{\Omega},\ t\in(0,T],
\end{equation}
with boundary condition $u(1,\theta,t)=0$ for $\theta\in[0,2\pi]$ and initial condition $u(r,\theta,0)=0$ for $(r,\theta)\in\tilde{\Omega}$. For the time discretization, we apply the backward Euler method on a uniform partition of $[0,T]$ with time step $\Delta t=T/N_T$, and denote by $u^n(r,\theta)\approx u(r,\theta,t_n)$ the approximation at $t_n=n\Delta t$. Then, for $n=1,2,\dots,N_T$, the time-discrete problem is
\begin{equation}\label{eq:time_discrete}
\frac{u^n-u^{n-1}}{\Delta t}
-\left(
\frac{1}{r}\frac{\partial}{\partial r}\left(r\frac{\partial u^n}{\partial r}\right)
+\frac{1}{r^2}\frac{\partial^2 u^n}{\partial \theta^2}
\right)
= \chi_{{}_{x\in D}}(r,\theta).
\end{equation}
Multiplying by a test function $v\in H_0^1(\tilde{\Omega})$ and integrating over $\tilde{\Omega}$ with respect to the measure $d\Omega=r\,dr\,d\theta$, we obtain the weak formulation
\begin{multline}\label{eq:weak_form}
\int_0^{2\pi}\int_0^1 u^n v\, r\,dr\,d\theta
+\Delta t \int_0^{2\pi}\int_0^1
\left(
r\frac{\partial u^n}{\partial r}\frac{\partial v}{\partial r}
+\frac{1}{r}\frac{\partial u^n}{\partial \theta}\frac{\partial v}{\partial \theta}
\right)\,dr\,d\theta \\
=
\int_0^{2\pi}\int_0^1 u^{n-1} v\, r\,dr\,d\theta
+\Delta t \int_0^{2\pi}\int_0^1 \chi_{{}_{x\in D}}\, v\, r\,dr\,d\theta.
\end{multline}

For the spatial discretization, we employ the finite element method on a mesh $\mathcal{T}_h$ consisting of rectangular elements, and choose a finite element space $V_h\subset H_0^1(\tilde{\Omega})$ of continuous piecewise quadratic polynomials. The discrete solution is sought in the form
\[
u_h^n(r,\theta)=\sum_{j=1}^{N_{dof}} U_j^n \phi_j(r,\theta),
\]
where $\{\phi_j\}$ is a basis of $V_h$. Substituting this expansion into \eqref{eq:weak_form} yields the linear system
\begin{equation}\label{eq:linear_system}
(M+\Delta t K)U^n = M U^{n-1} + \Delta t F,
\end{equation}
where
\begin{align*}
M_{ij} &= \int_0^{2\pi}\int_0^1 \phi_j\phi_i\, r\,dr\,d\theta,\\
K_{ij} &= \int_0^{2\pi}\int_0^1
\left(
r\frac{\partial \phi_j}{\partial r}\frac{\partial \phi_i}{\partial r}
+\frac{1}{r}\frac{\partial \phi_j}{\partial \theta}\frac{\partial \phi_i}{\partial \theta}
\right)\,dr\,d\theta,\\
F_i &= \int_0^{2\pi}\int_0^1 \chi_{{}_{x\in D}}\,\phi_i\, r\,dr\,d\theta.
\end{align*}

\subsection{Bayesian framework.}\label{subsec:inference}

The inference step in Algorithm~\ref{alg:adaptive_placement} is carried out within a Bayesian framework using an adaptive pCN-MCMC sampler. This approach provides not only point estimates of the unknown source parameters $\boldsymbol{\xi}$ but also a quantitative characterization of their uncertainty as data are collected sequentially. Let $\mathbf d\in\mathbb{R}^{n_d}$ denote the vector of flux measurements, and let $\boldsymbol{\xi}$ be the unknown parameter vector with prior distribution $p_0(\boldsymbol{\xi})$. By Bayes' theorem, the posterior distribution is given by
\begin{equation}
p(\boldsymbol{\xi}\mid \mathbf d) \propto p_0(\boldsymbol{\xi})\,p(\mathbf d\mid \boldsymbol{\xi}),
\end{equation}
where, under the assumption of additive Gaussian measurement noise with variance $\sigma^2$, the likelihood takes the form
\begin{equation}
p(\mathbf d\mid \boldsymbol{\xi})
=
(2\pi\sigma^2)^{-n_d/2}
\exp\left(
-\frac{\|\mathbf d-g(\boldsymbol{\xi})\|_2^2}{2\sigma^2}
\right),
\end{equation}
with $g$ denoting the forward map from the parameter space to the observation space.

To sample from this posterior distribution, we employ an adaptive preconditioned Crank--Nicolson Markov chain Monte Carlo (pCN-MCMC) method \cite{Cotter2013, Hu2017}, which is well suited for high-dimensional inverse problems and does not require gradient information. For a Gaussian prior $\mathcal N(0,B)$, the standard pCN proposal is
\begin{equation}\label{eq:pcn_proposal}
\boldsymbol{\xi}^*
=
\sqrt{1-\beta_1^2}\,\boldsymbol{\xi}
+\beta_1 \boldsymbol{w}_1,
\qquad
\boldsymbol{w}_1\sim \mathcal N(0,B).
\end{equation}
To improve sampling efficiency, we further adopt an adaptive version in which the prior covariance is replaced by an empirical approximation $C$ of the posterior covariance computed from previous samples:
\begin{equation}\label{eq:adaptive_pcn_proposal}
\boldsymbol{\xi}^*
=
\bigl(I-\beta_2^2 C B^{-1}\bigr)^{1/2}\boldsymbol{\xi}
+\beta_2 \boldsymbol{w}_2,
\qquad
\boldsymbol{w}_2\sim \mathcal N(0,C).
\end{equation}
The proposed state is accepted with probability
\[
\alpha(\boldsymbol{\xi}^*,\boldsymbol{\xi})
=
\min\left\{
1,\,
\exp\bigl(\Phi(\boldsymbol{\xi})-\Phi(\boldsymbol{\xi}^*)\bigr)
\right\},
\]
where
\[
\Phi(\boldsymbol{\xi})
=
\frac{1}{2\sigma^2}\|\mathbf d-g(\boldsymbol{\xi})\|_2^2
\]
is the data misfit functional, namely, the negative log-likelihood up to an additive constant.
This adaptive pCN procedure serves as the subroutine $\mathrm{Infer}(\cdot)$ in Algorithm~\ref{alg:adaptive_placement}, and its implementation is summarized in Algorithm~\ref{alg:adaptive-pcn-mcmc}.

\begin{algorithm} [!htbp]
\caption{Adaptive pCN MCMC (Inference Subroutine)}\label{alg:adaptive-pcn-mcmc}
\begin{algorithmic}[1]
\Require Prior covariance $B$, parameters $\beta_1,\beta_2$, noise level $\sigma$, initial state $\boldsymbol{\xi}_0$, observation time $\bar{t}$, sample counts $N_1$ and $N$, update frequency $k_0$
\Ensure Posterior samples $\{\boldsymbol{\xi}_i\}_{i=1}^{N}$ and corresponding PDE solutions at $\bar{t}$
\Statex
\For{$k=1$ to $N_1$}
    \State Propose $\boldsymbol{\xi}^*$ using the pCN scheme \eqref{eq:pcn_proposal}
    \State Accept $\boldsymbol{\xi}^*$ as $\xi_k$ with probability $\alpha(\boldsymbol{\xi}^*,\boldsymbol{\xi}_{k-1})$
\EndFor
\State Collect samples $\{\boldsymbol{\xi}_i\}_{i=1}^{N_1}$ and compute the empirical covariance $C$
\Statex
\For{$k=N_1+1$ to $N$}
    \If{$\operatorname{mod}(k,k_0+1)=0$}
        \State Update $C$ using all samples $\{\boldsymbol{\xi}_i\}_{i=1}^{k}$
    \EndIf
    \State Propose $\boldsymbol{\xi}^*$ using the adaptive pCN scheme \eqref{eq:adaptive_pcn_proposal}
    \State Accept $\boldsymbol{\xi}^*$ as $\boldsymbol{\xi}_k$ with probability $\alpha(\boldsymbol{\xi}^*,\boldsymbol{\xi}_{k-1})$
    \State Store $\boldsymbol{\xi}_k$ and the corresponding PDE solution at $\bar{t}$
\EndFor
\end{algorithmic}
\end{algorithm}

\subsection{Movement strategy of sensor.}

When the source location is poorly known a priori, a static sensor arrangement may easily produce measurements with limited information content, resulting in inefficient or unstable reconstruction. To overcome this difficulty, we introduce a closed-loop \emph{Measure--Infer--Move} strategy that movingally repositions a mobile sensor along $\partial\Omega$ on the basis of the observed boundary flux. The guiding principle is simple: by exploiting local variations of the boundary heat flux $\partial_n u$, the sensor can be steered toward boundary regions that are more sensitive to the unknown source, thereby producing more informative measurements under a fixed sampling budget.

To this end, the observation horizon is divided into successive time windows $\{[T_k,T_{k+1}]\}$. Within each window, the procedure consists of three steps:
\begin{enumerate}
    \item \textbf{Measure:} sample the boundary heat flux at the current sensor location $z^{(k)}$ at discrete times in the current window;
    \item \textbf{Infer:} update the accumulated data set $\mathcal D$ and infer the source parameters $\boldsymbol{\xi}$ using the Bayesian procedure described above;
    \item \textbf{Move:} at the end of the window, evaluate both the local flux and its angular derivative at $z^{(k)}$, and use this information to determine the next direction and step size of the sensor along the boundary.
\end{enumerate}

Intuitively, people would prefer the measurements with high magnitude, which could be written as $|\frac{\partial u}{\n}|$. Sequentially the movement strategy of sensor should ensure that the sensor moves from the location with low magnitude $|\frac{\partial u}{\n}|$ to the one with high magnitude. Since the path of sensor is contained by the boundary $\partial \Omega$, we need to use the information of tangential partial derivative of $|\frac{\partial u}{\n}|$, stated as $\frac{\partial}{\partial \tau}|\frac{\partial u}{\n}|$ and $\tau$ is the tangential vector of boundary. Fortunately, in this work $\Omega$ is set as the unit disc in $\mathbb R^2$. Hence, by the polar formulation \eqref{eq:polar_pde}, we have 
$$\frac{\partial}{\partial \tau}\Big(\big|\frac{\partial u}{\n}\big|\Big)=\frac{\partial}{\partial \theta}\Big(\big|\frac{\partial u}{\n}\big|\Big).$$
Precisely, we could introduce the movement strategy as follows: 
\begin{equation}\label{movement}
\text{measure}\ \Phi^{(k)}_\theta := \frac{\partial}{\partial \theta}\Big(\big|\frac{\partial u}{\n}(z^{(k)},T_{k+1})\big|\Big),
\quad
\text{then set}\ \mathrm{dir}=
\begin{cases}
\text{counterclockwise}, & \text{if } \Phi^{(k)}_\theta>0,\\
\text{clockwise}, & \text{if } \Phi^{(k)}_\theta\le 0.
\end{cases}
\end{equation}
Here 'dir' is the abbreviation of 'direction'. 
This strategy could drive the sensor toward locations where $| \frac{\partial u}{\n}| $ increases along the boundary. The step size $d_k$ and the corresponding travel time $b_{k+1}$ are defined by
\begin{equation}\label{eq:step_rule}
d_k=
\begin{cases}
mc_1\pi, & \mathrm{prev}=\varnothing\ \text{or}\ \mathrm{dir}=\mathrm{prev},\\
\lfloor m/2\rfloor c_1\pi, & \text{otherwise},
\end{cases}
\qquad
b_{k+1}=d_k/c,
\end{equation}
where $\mathrm{prev}$ records the previous direction, $c$ is the sensor speed, and $m,c_1$ are tuning parameters.

The iteration stops either when the sensor reaches a strict local maximum of $| \frac{\partial u}{\n}| $ on the boundary, or when a reversal of the movement direction is detected, in which case one final inference step is performed using the newly acquired data. The complete procedure is summarized in Algorithm~\ref{alg:adaptive_placement}.

\begin{algorithm}[!htbp]
\caption{Movement Strategy for Source Identification}
\label{alg:adaptive_placement}
\footnotesize
\begin{algorithmic}[1]
\Require Time partition $\{T_k\}_{k=0}^{n+1}$, speed $c$, parameters $m,c_1$, samples per window $N_t$, initial point $z^{(0)}\in\partial\Omega$
\State $k\gets 0;\quad b_0\gets 0;\quad \mathcal D\gets \emptyset;\quad \mathrm{prev}\gets \varnothing$
\While{true}
    \State Choose a discrete sampling set
    \[
    d_t^{(k)}=\{t_j^{(k)}\}_{j=1}^{N_t}\subset [T_k+b_k,\;T_{k+1}]
    \]
    such that $T_{k+1}\in d_t^{(k)}$ (e.g.\ uniformly spaced)

    \State \textbf{Measure:} Solve the forward problem \eqref{eq:linear_system} with the true parameter $\boldsymbol{\xi}_{\text{true}}$ to obtain the numerical solution $u_h$. Then collect the noisy flux measurements at the sensor location $z^{(k)}$ over the time set $d_t^{(k)}$:
    \[
    \mathbf d^{(k)}
    :=
    \left[
    \frac{\partial u_h}{\partial \vec n}(z^{(k)},t)
    \right]_{t\in d_t^{(k)}}
    +
    \bm\epsilon^{(k)},
    \qquad
    \bm\epsilon^{(k)}\sim\mathcal N(\mathbf 0,\sigma^2 I_{N_t}).
    \]

    \State $\mathcal D \gets \mathcal D\cup \{(d_t^{(k)},\mathbf d^{(k)})\}$

    \State \textbf{Infer:} $\boldsymbol\xi^{(k)} \gets \mathrm{Infer}(\mathcal D)$ \Comment{Uses Algorithm~\ref{alg:adaptive-pcn-mcmc}}

    \State At time $T_{k+1}$, measure the noisy boundary flux at $z^{(k)}$ and at its two neighboring boundary points $z_-^{(k)}$ and $z_+^{(k)}$:
    \[
    d_{\mathrm{end}}^{(k)}(z)
    :=
    \frac{\partial u_h}{\partial \vec n}(z,T_{k+1})
    +
    \epsilon_{\mathrm{end}}^{(k)}(z),
    \qquad
    z\in\{z_-^{(k)},\,z^{(k)},\,z_+^{(k)}\},
    \]
    \[
    \epsilon_{\mathrm{end}}^{(k)}(z)\sim\mathcal N(0,\sigma^2).
    \]

    \State Compute
    \[
    \frac{\partial}{\partial \theta}
    \left(
    \left|d_{\mathrm{end}}^{(k)}(z^{(k)})\right|
    \right)
    \approx
    \frac{
    \left|d_{\mathrm{end}}^{(k)}(z_+^{(k)})\right|
    -
    \left|d_{\mathrm{end}}^{(k)}(z_-^{(k)})\right|
    }{2\Delta\theta}
    \]
    at $z^{(k)}$ by a local finite difference

    \State $\mathrm{dir}\gets \mathrm{Direction}\!\left(
    \frac{\partial}{\partial \theta}
    \left(
    \left|d_{\mathrm{end}}^{(k)}(z^{(k)})\right|
    \right)
    \right)$ via \eqref{movement}

    \State Compute $(d_k,b_{k+1})$ from \eqref{eq:step_rule}

    \State \textbf{Move:} Update $z^{(k+1)}$ by moving from $z^{(k)}$ along $\partial\Omega$ in direction $\mathrm{dir}$ over distance $d_k$

    \State $\mathrm{rev}\gets(\mathrm{prev}\neq\varnothing\ \wedge\ \mathrm{dir}\neq \mathrm{prev});\quad \mathrm{prev}\gets\mathrm{dir}$

    \If{$\left|d_{\mathrm{end}}^{(k)}(z^{(k)})\right|$ is a strict local maximum among $\left|d_{\mathrm{end}}^{(k)}(z_-^{(k)})\right|$, $\left|d_{\mathrm{end}}^{(k)}(z^{(k)})\right|$, and $\left|d_{\mathrm{end}}^{(k)}(z_+^{(k)})\right|$ at time $T_{k+1}$}
        \State \Return $\boldsymbol\xi^{(k)}$
    \ElsIf{$\mathrm{rev}$}
        \State Take one additional sensing window on $[T_{k+1}+b_{k+1},\;T_{k+2}]$ and update $\mathcal D$
        \State \Return $\mathrm{Infer}(\mathcal D)$
    \Else
        \State $k\gets k+1$
    \EndIf
\EndWhile
\end{algorithmic}
\end{algorithm}

\section{Numerical experiments}\label{sec_experiment}

We validate the proposed adaptive algorithm through a series of numerical experiments. Synthetic ``ground truth'' data are generated by solving the forward problem on a fine spatial grid of $23\times 23$ nodes with time step $\Delta t=1/400$. To avoid the inverse crime, all inversion procedures are carried out on a coarser $20\times 20$ grid.

Several representative geometries are considered for the unknown source domain $D$, each parameterized by a vector $\boldsymbol{\xi}$. The source term is described by the characteristic function
\[
\chi_{{}_{x\in D}}(x)=
\begin{cases}
b, & \text{if } x\in D(\boldsymbol{\xi}),\\[2pt]
0, & \text{otherwise},
\end{cases}
\]
where $b$ denotes the source strength.

\begin{example}[Circular Source]\label{examcircle}
A circular source is defined in Cartesian coordinates by
\begin{align*}
x_1 &= \xi_1\cos\xi_2+\xi_3\cos\theta,\\
x_2 &= \xi_1\sin\xi_2+\xi_3\sin\theta,\qquad \theta\in[0,2\pi],
\end{align*}
where the unknown parameter vector is $\boldsymbol{\xi}=(\xi_1,\xi_2,\xi_3)$. Here, $(\xi_1,\xi_2)$ represent the polar coordinates of the centre, and $\xi_3$ denotes the radius.
\end{example}

Within the Bayesian framework, sampling is performed in an unconstrained space. We introduce reparameterization variables $\boldsymbol{z}=(z_1,z_2,z_3)\in\mathbb{R}^3$ and define the physical parameters by
\begin{align*}
\xi_1(z_1) &= \frac{1}{\pi}\arctan z_1+\frac12
&&\text{(maps to }(0,1)),\\[2pt]
\xi_2(z_2) &= 2\arctan z_2+\pi
&&\text{(maps to }(0,2\pi)),\\[2pt]
\xi_3(z_3) &= \frac{1}{\pi}\arctan z_3+\frac12
&&\text{(maps to }(0,1)).
\end{align*}
A Gaussian prior $\mathcal{N}(\boldsymbol{0},I_3)$ is assigned to $\boldsymbol{z}$. Posterior inference for $\boldsymbol{\xi}$ is then performed by sampling $\boldsymbol{z}$ and applying the above transformations.

For this example, the source strength is fixed at $b=50$. The true parameters are taken as
\[
\boldsymbol{\xi}_{\text{true}}=(0.7,\pi/2,0.2)^\top,
\]
which correspond to a circle centred at polar coordinates $(0.7,\pi/2)$ with radius $0.2$. The measurement noise variance is set to $\sigma^2=0.05^2$, corresponding to an additive absolute error with standard deviation $0.05$. The adaptive pCN-MCMC algorithm is run with $N_1=0$ burn-in iterations and $N=10^4$ total iterations, and the empirical covariance matrix is updated every $k_0=2.5\times10^3$ iterations. The step-size parameters $\beta_1$ and $\beta_2$ are tuned so that the acceptance rate remains between $25\%$ and $35\%$. For Algorithm~\ref{alg:adaptive_placement}, we set $m=10$, $c_1=1/20$, $c=20\pi$, and $N_t=80$.

Figure~\ref{circle} illustrates the evolution of the reconstruction over four successive time windows, where the posterior mean is plotted in blue. The results show that the proposed strategy is able to recover the circular source rapidly and stably from sparse observations, and that the reconstruction improves progressively as the sensor location is updated.

\begin{figure}[!htbp]
\centering
\subfigure[{$T\in[1/400,80/400]$, sensor at $26\pi/40$.}]{\includegraphics[width=0.38\textwidth]{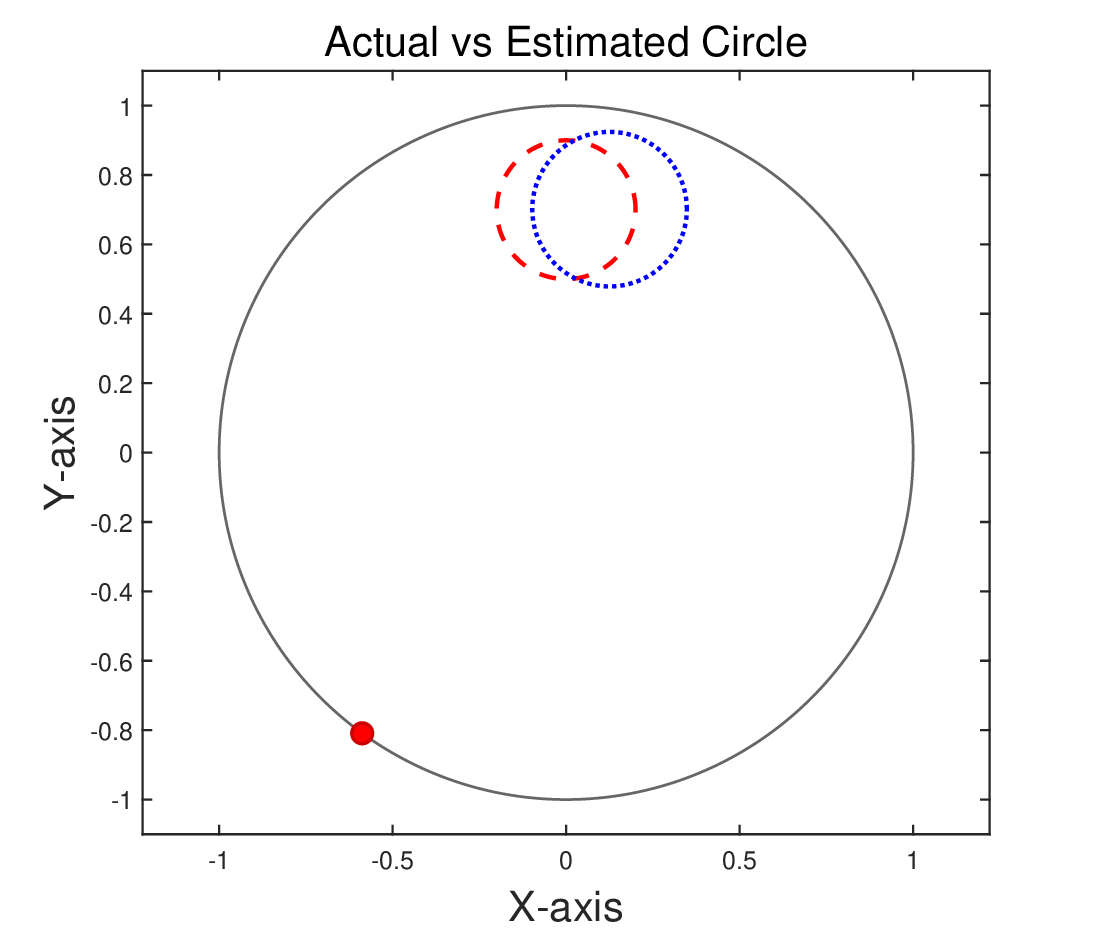}}
\subfigure[{$T\in[90/400,170/400]$, sensor at $16\pi/40$.}]{\includegraphics[width=0.38\textwidth]{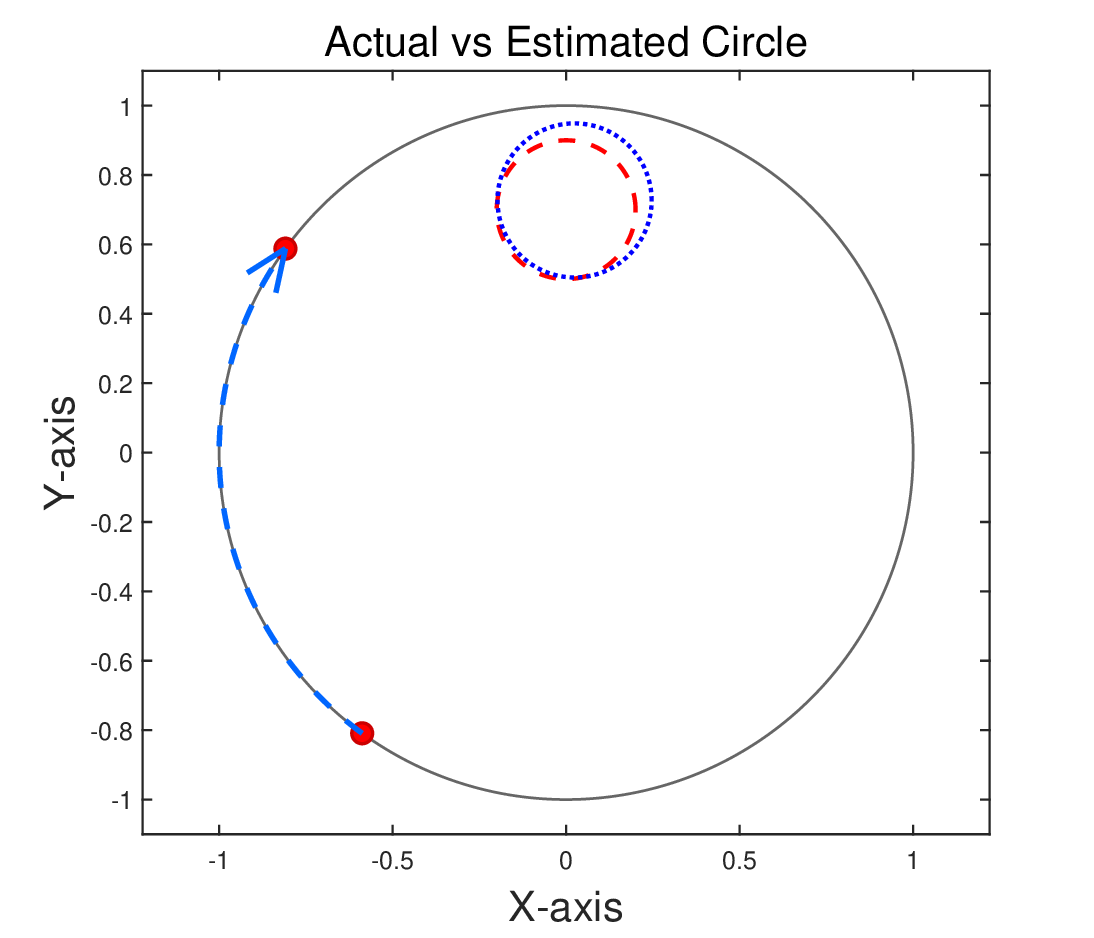}}
\subfigure[{$T\in[180/400,260/400]$, sensor at $6\pi/40$.}]{\includegraphics[width=0.38\textwidth]{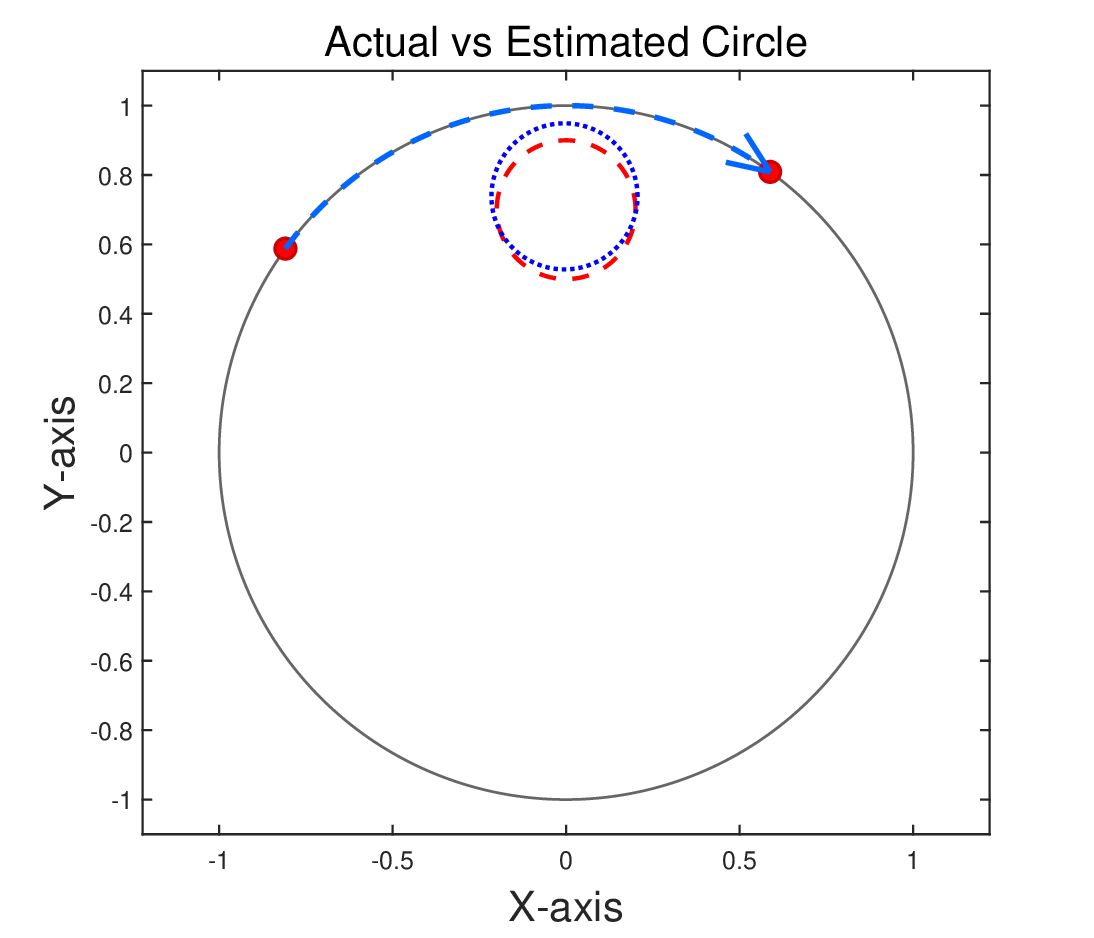}}
\subfigure[{$T\in[265/400,345/400]$, sensor at $11\pi/40$.}]{\includegraphics[width=0.38\textwidth]{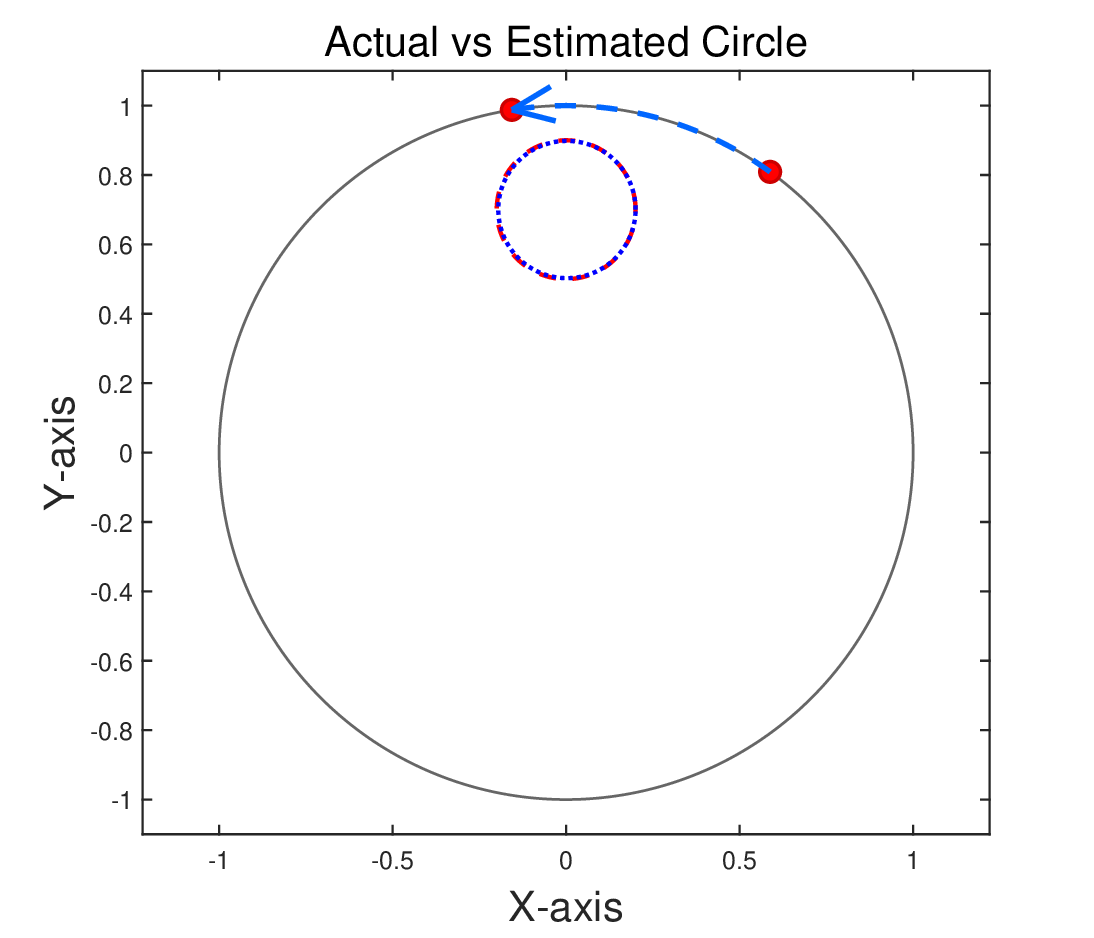}}
\caption{Evolution of the reconstruction for a circular source with $N=10000$.}
\label{circle}
\end{figure}

To further assess the impact of the moving sensor strategy, Figure~\ref{fig:bayecircle} displays the trace plots of the posterior samples for $\xi_1$, $\xi_2$, and $\xi_3$ at the initial and final stages, respectively. A pronounced reduction in sample dispersion is observed in the final stage, indicating that the sequentially selected sensor locations provide more informative measurements and substantially reduce posterior uncertainty.

\begin{figure}[!htbp]
\centering
\subfigure[{Corresponding to Fig.~\ref{circle}(a)}]{\includegraphics[width=0.38\textwidth]{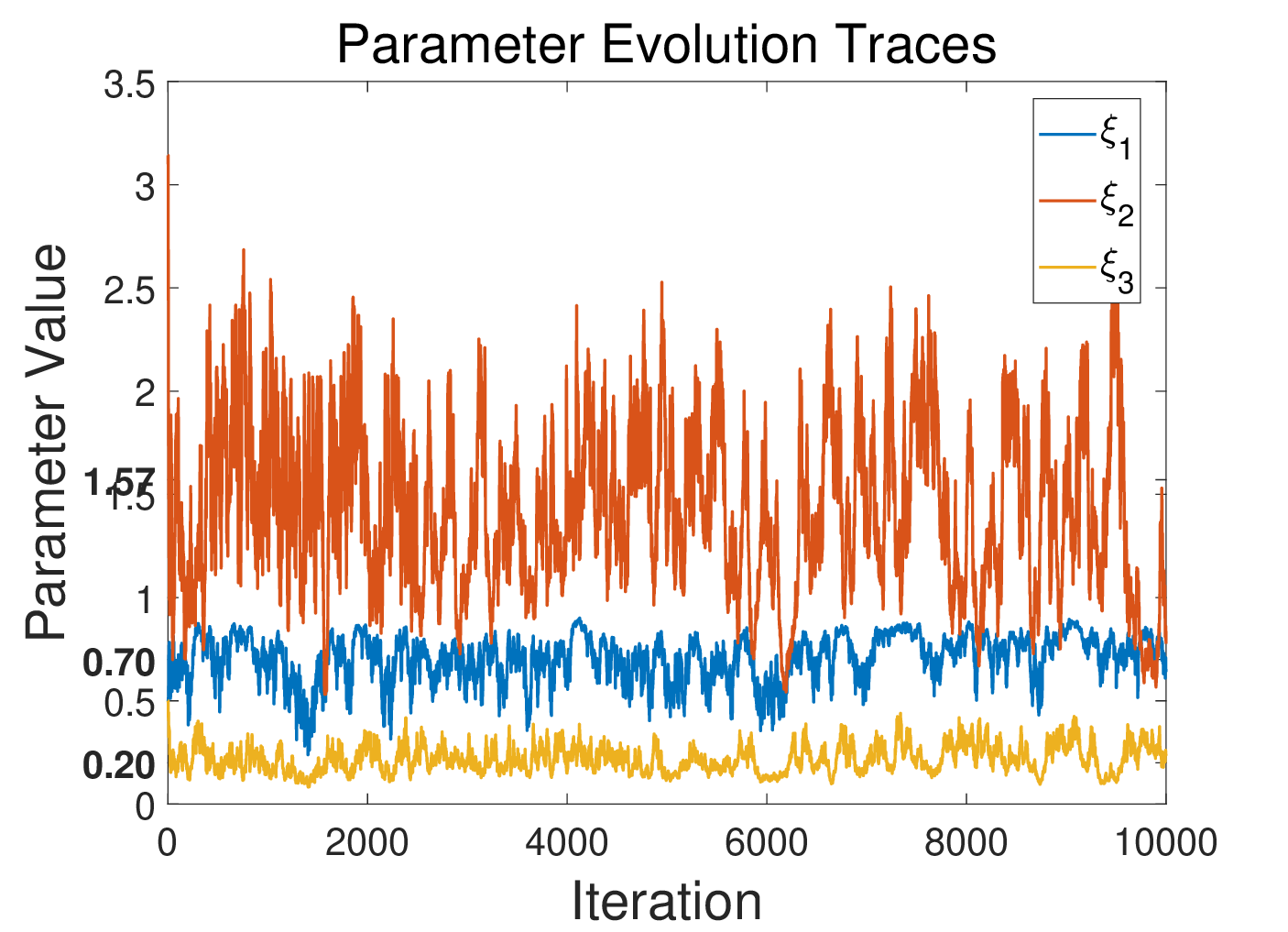}}
\subfigure[{Corresponding to Fig.~\ref{circle}(d)}]{\includegraphics[width=0.38\textwidth]{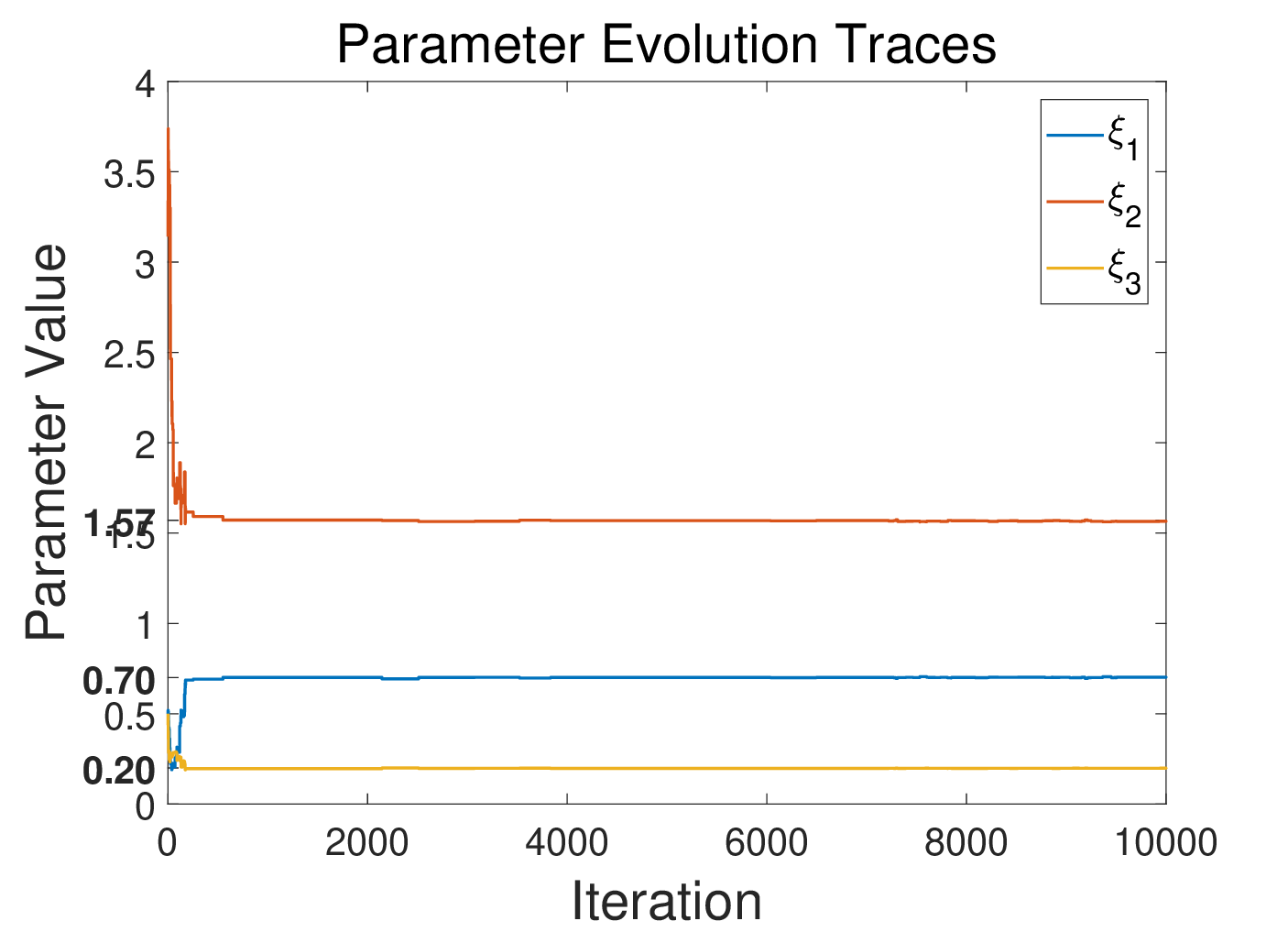}}
\caption{Trace plots of the posterior samples for the circular source with $N=10000$. The true parameters are $\boldsymbol{\xi}_{\text{true}}=(0.7,\pi/2,0.2)^\top$. The moving sensor placement strategy (right figure) leads to markedly narrower confidence bands, signifying reduced sample variance.}
\label{fig:bayecircle}
\end{figure}

We next examine the performance under a reduced sampling budget with only $N=500$ iterations. As shown in Figure~\ref{fig:bayecircleacn}, the moving sensor strategy still yields a noticeably more accurate approximation than the fixed-sensor counterpart, further demonstrating the efficiency of the proposed approach.

\begin{figure}[!htbp]
\centering
\subfigure[{Corresponding to Fig.~\ref{circle}(a)}]{\includegraphics[width=0.38\textwidth]{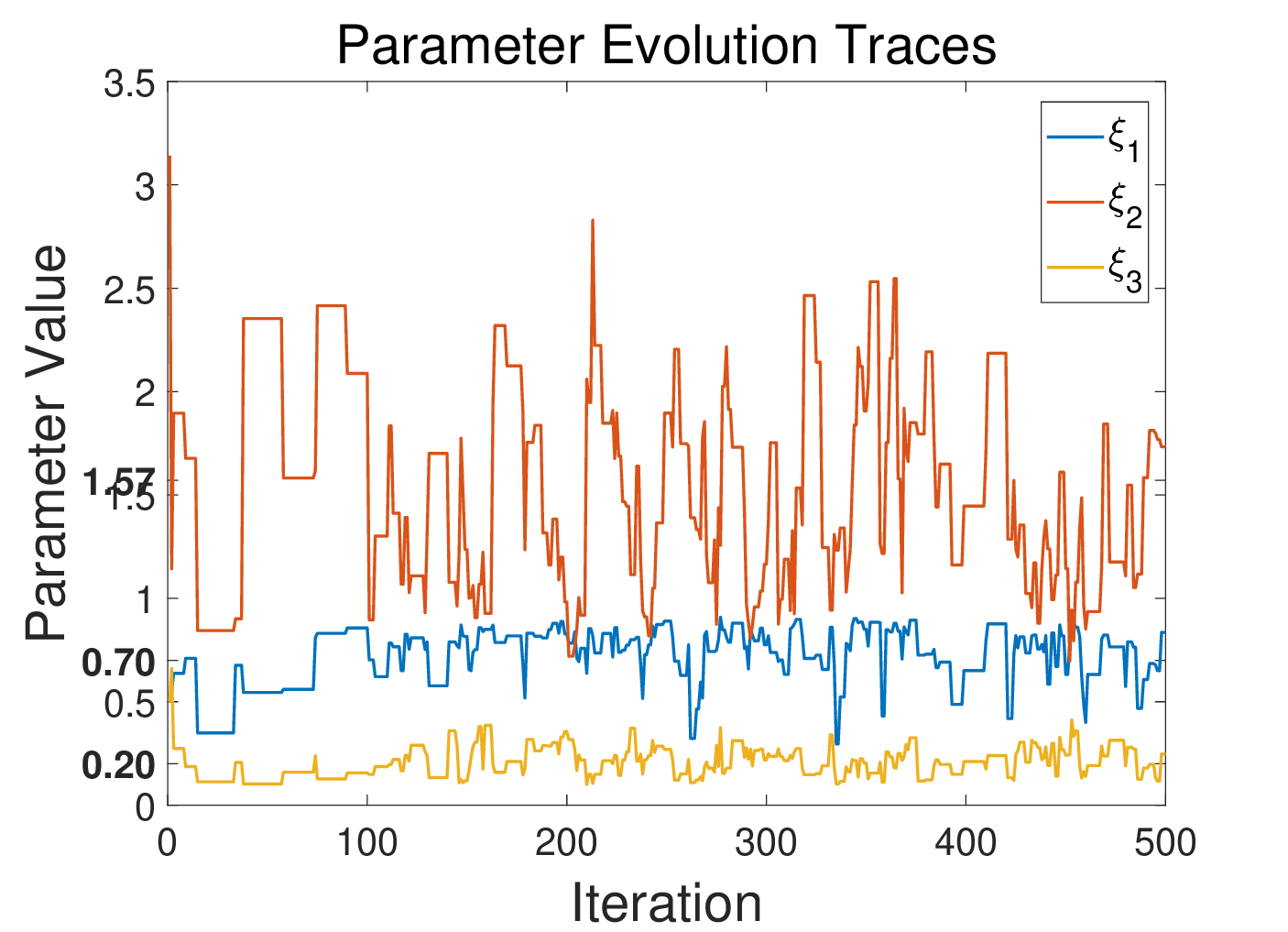}}
\subfigure[{Corresponding to Fig.~\ref{circle}(d)}]{\includegraphics[width=0.38\textwidth]{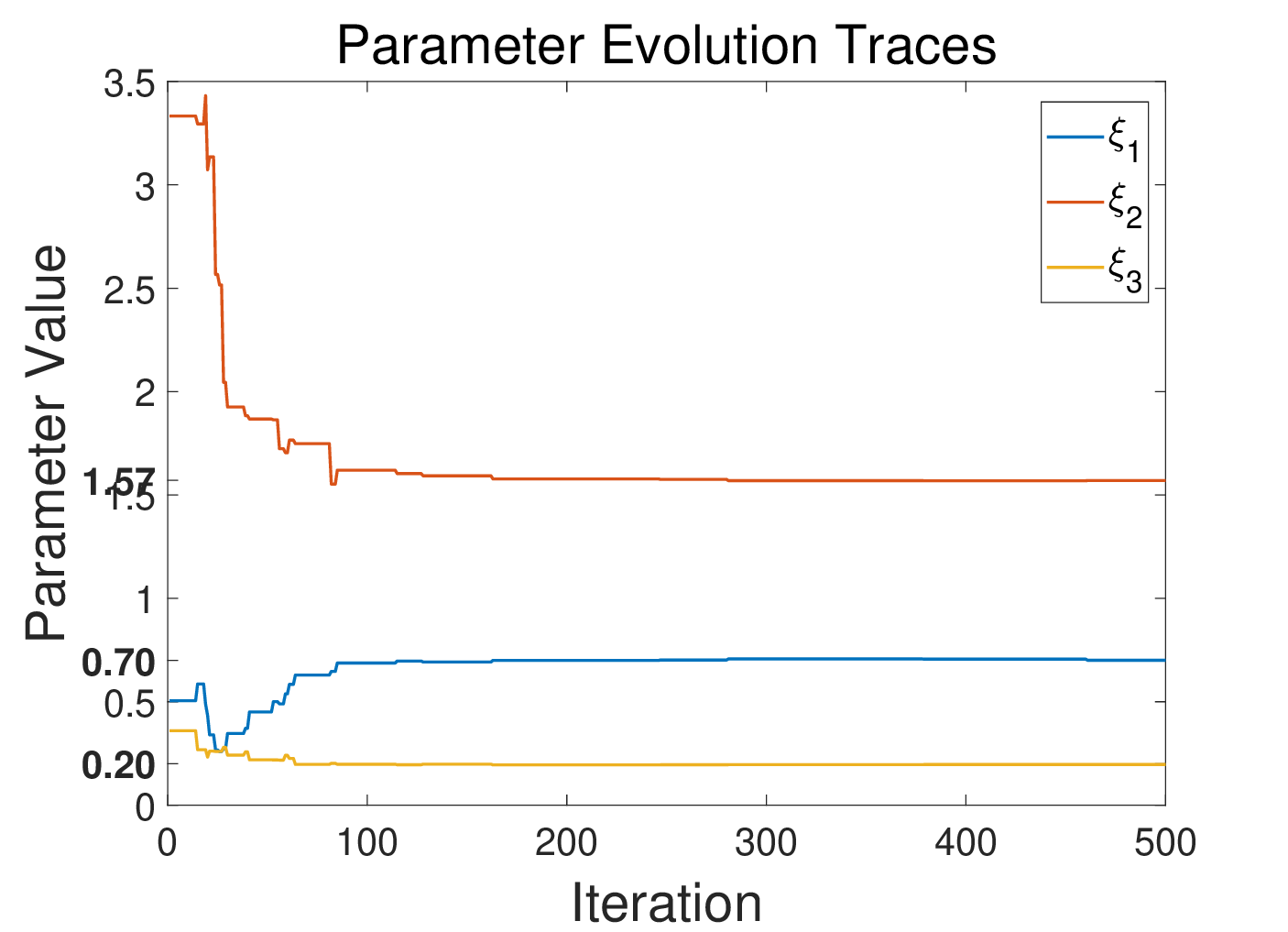}}
\caption{Trace plots of the posterior samples for the circular source. The true parameters are $\boldsymbol{\xi}_{\text{true}}=(0.7,\pi/2,0.2)^\top$. With only $N=500$ iterations, the moving sensor placement strategy (right figure) yields a more accurate approximation compared to the left figure.}
\label{fig:bayecircleacn}
\end{figure}

\FloatBarrier

\begin{example}[Kite-Shaped Source]
The boundary of a kite-shaped source is described by
\begin{align*}
x_1 &= \xi_1\cos\xi_2+\xi_3\bigl[\cos\theta+0.65\cos 2\theta-0.65\bigr],\\
x_2 &= \xi_1\sin\xi_2+1.5\,\xi_3\sin\theta,\qquad \theta\in[0,2\pi],
\end{align*}
with unknown parameters $\boldsymbol{\xi}=(\xi_1,\xi_2,\xi_3)$. The same bijective mappings as in Example~\ref{examcircle} are applied to map the unconstrained variables to the physical intervals $(0,1)$, $(0,2\pi)$, and $(0,1)$, respectively. A Gaussian prior $\mathcal{N}(\boldsymbol{0},I_3)$ is imposed on the transformed variables.
\end{example}

The true parameters are chosen as
\[
\boldsymbol{\xi}_{\text{true}}=(0.4,\pi/3,0.2)^\top,
\]
which correspond to a kite centred at polar coordinates $(0.4,\pi/3)$ with scaling factor $0.2$. We set $b=50$ and $\sigma^2=0.05^2$. The adaptive pCN-MCMC algorithm is run with $N_1=0$, $N=10^4$, and $k_0=2.5\times10^3$, with step sizes tuned to maintain an acceptance rate between $25\%$ and $35\%$. In Algorithm~\ref{alg:adaptive_placement}, we take $m=10$, $c_1=1/20$, $c=20\pi$, and $N_t=80$.

Figure~\ref{kite} presents the evolution of the reconstruction for the kite-shaped source. The posterior mean gradually approaches the true geometry as additional data are collected, showing that the proposed method remains effective for this non-convex shape and is able to exploit sparse boundary measurements efficiently.

\begin{figure}[!htbp]
\centering
\subfigure[{$T\in[1/400,80/400]$, sensor at $26\pi/40$.}]{\includegraphics[width=0.38\textwidth]{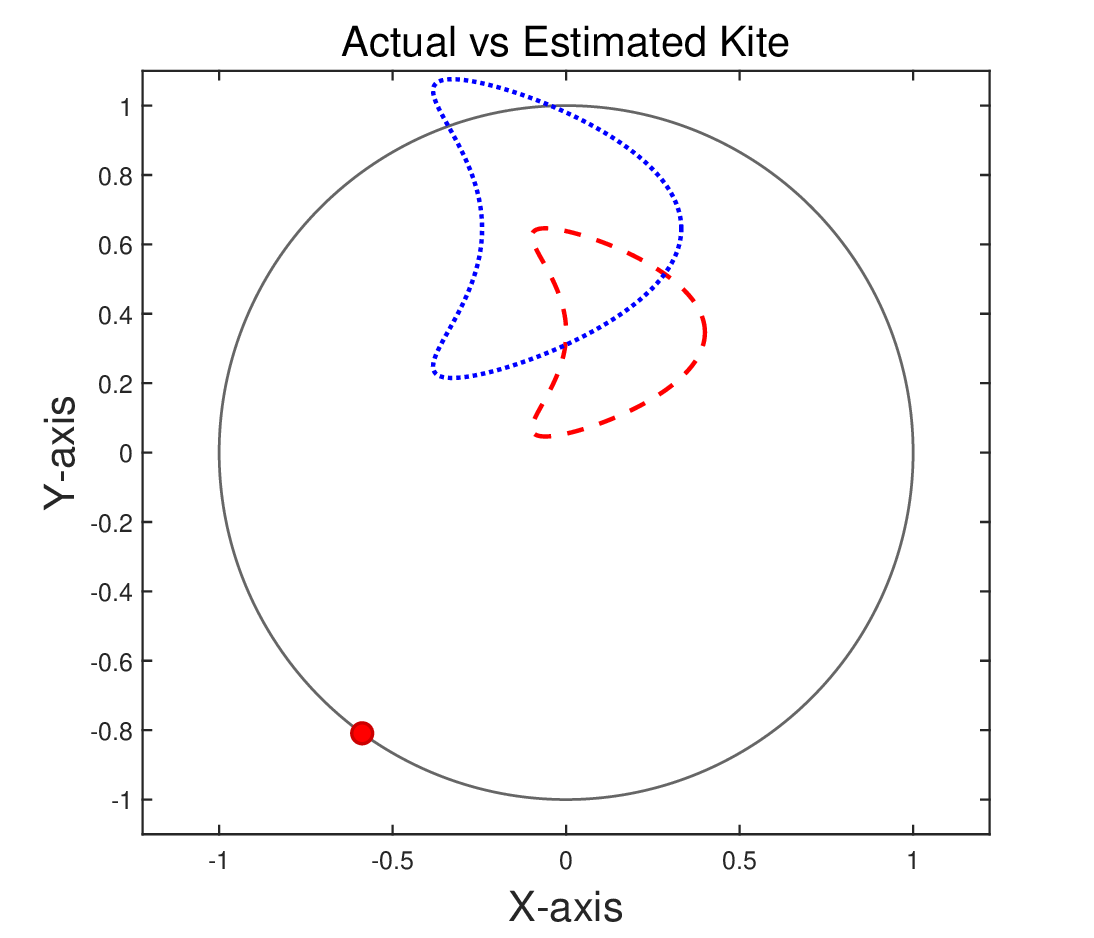}}
\subfigure[{$T\in[90/400,170/400]$, sensor at $16\pi/40$.}]{\includegraphics[width=0.38\textwidth]{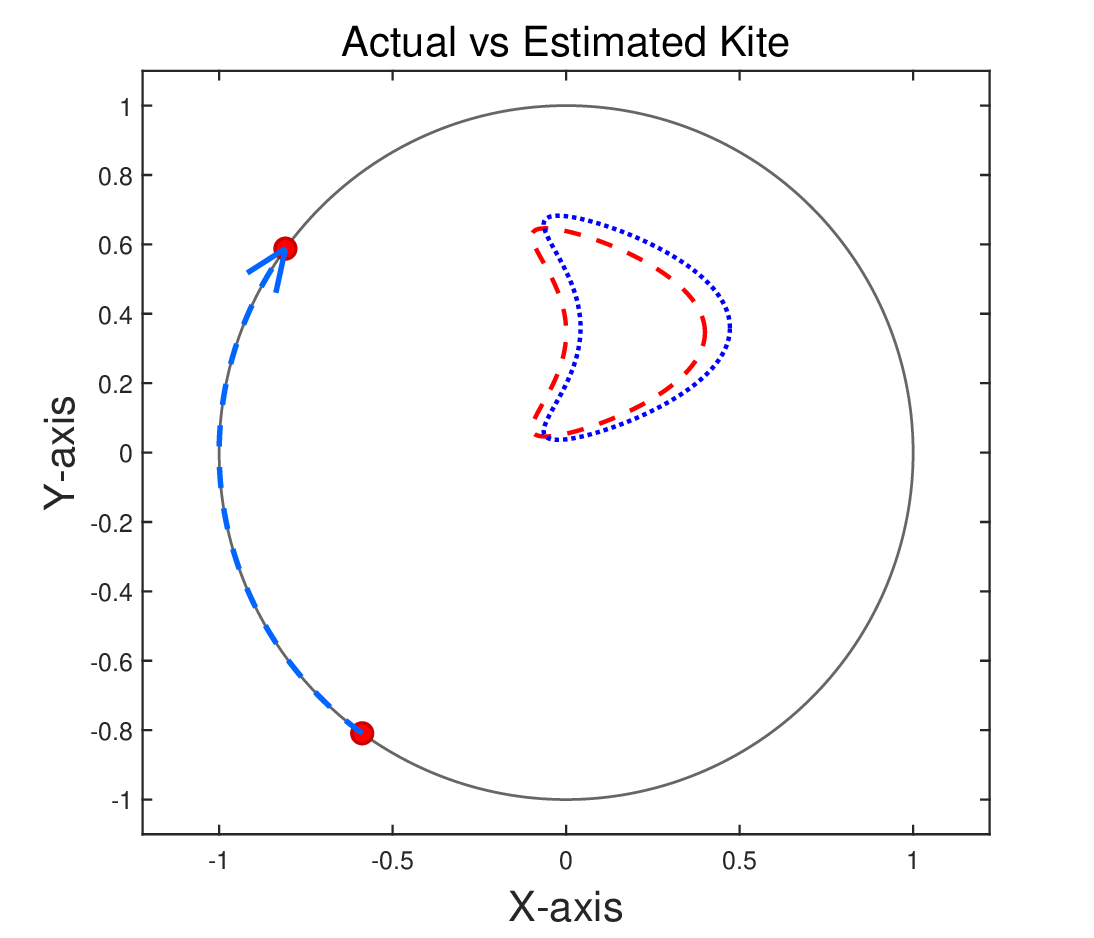}}
\subfigure[{$T\in[180/400,260/400]$, sensor at $6\pi/40$.}]{\includegraphics[width=0.38\textwidth]{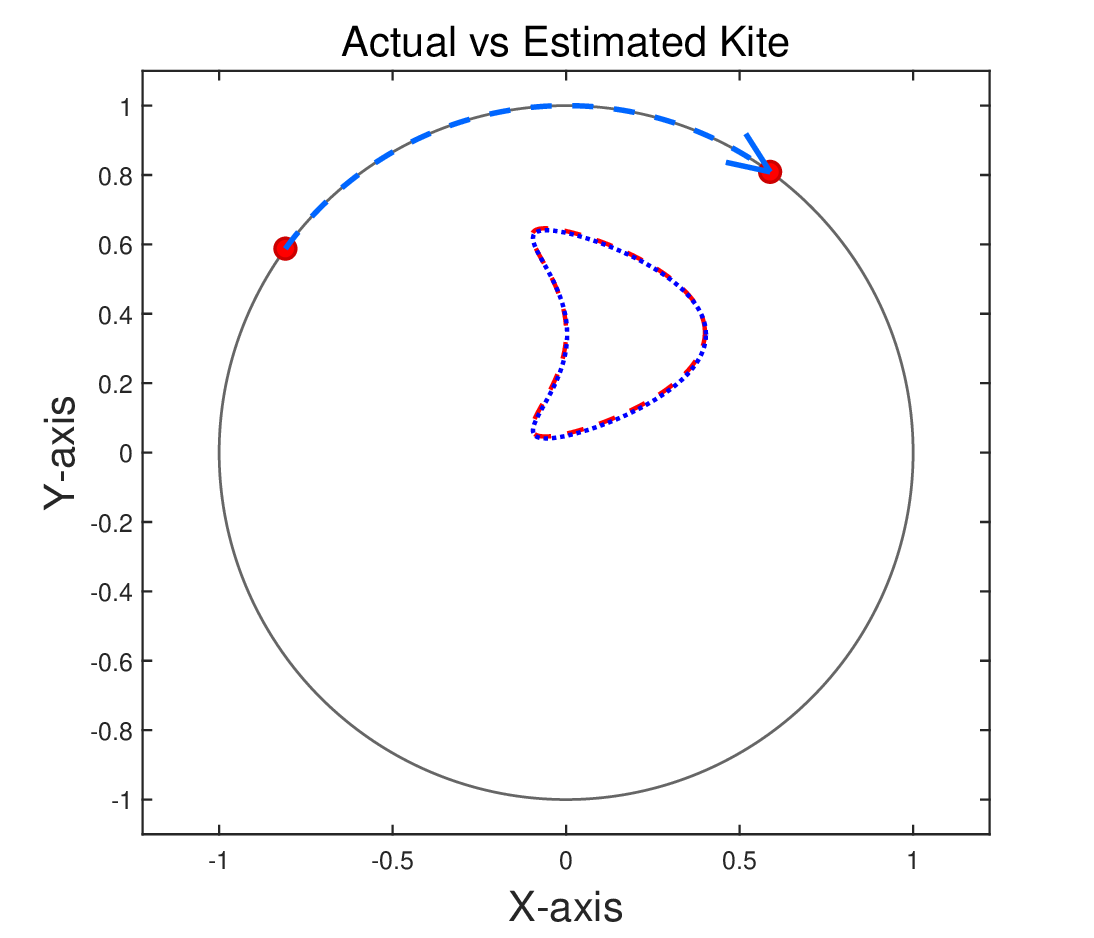}}
\subfigure[{$T\in[265/400,345/400]$, sensor at $11\pi/40$.}]{\includegraphics[width=0.38\textwidth]{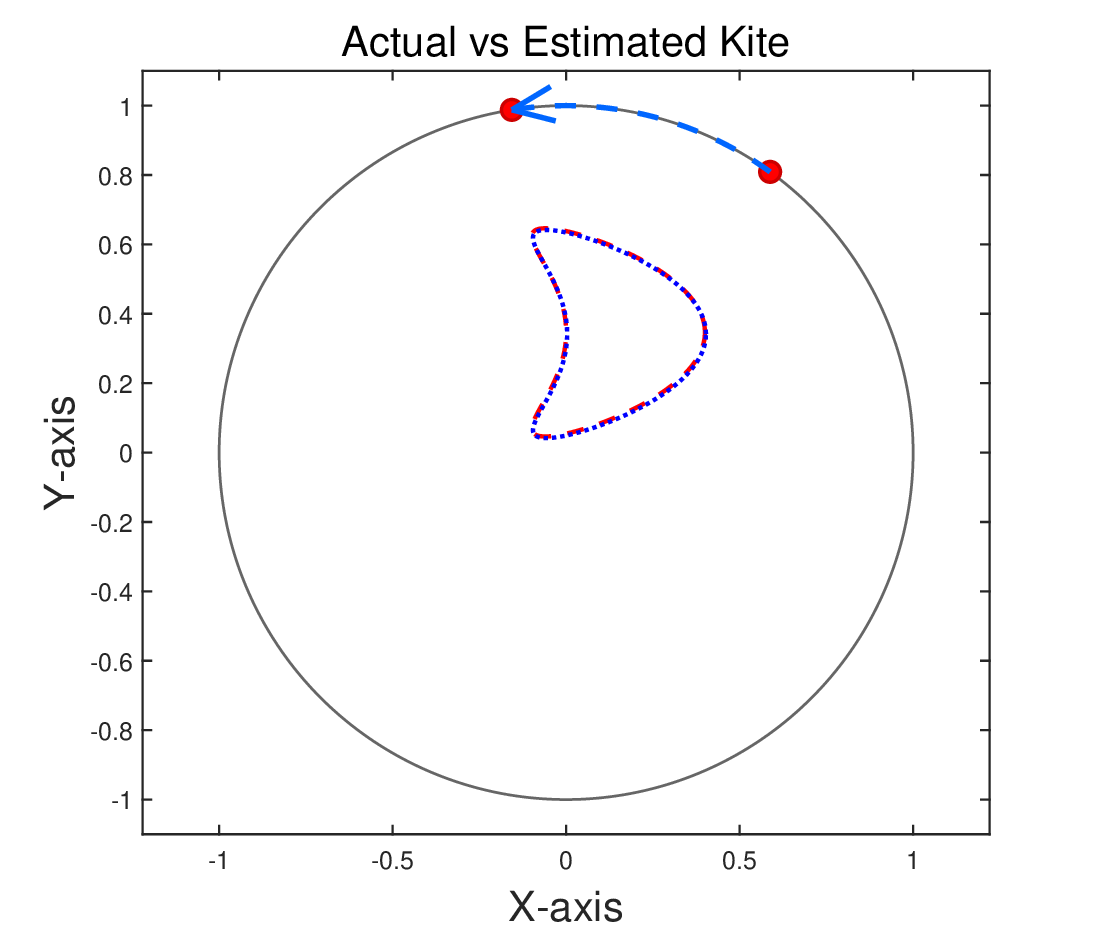}}
\caption{Evolution of the reconstruction for a kite-shaped source with $N=10000$.}
\label{kite}
\end{figure}

The corresponding posterior trace plots are shown in Figure~\ref{fig:bayekite}. Compared with the initial fixed-sensor stage, the final moving-sensor stage yields substantially narrower sample bands for all three parameters. This indicates that the adaptive placement mechanism improves the informativeness of the measurements and enhances the concentration of the posterior distribution.

\begin{figure}[!htbp]
\centering
\subfigure[{Corresponding to Fig.~\ref{kite}(a)}]{\includegraphics[width=0.38\textwidth]{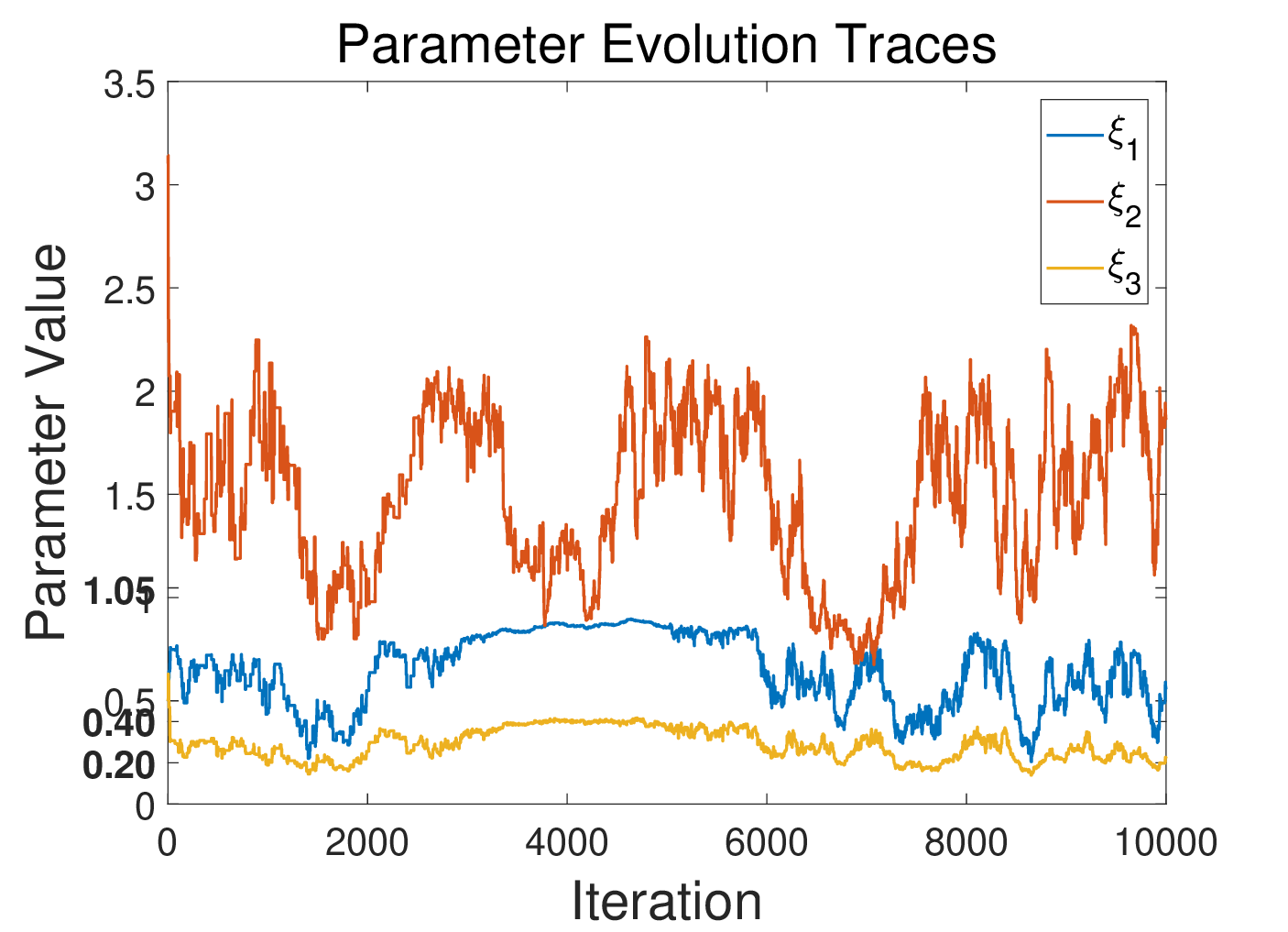}}
\subfigure[{Corresponding to Fig.~\ref{kite}(d)}]{\includegraphics[width=0.38\textwidth]{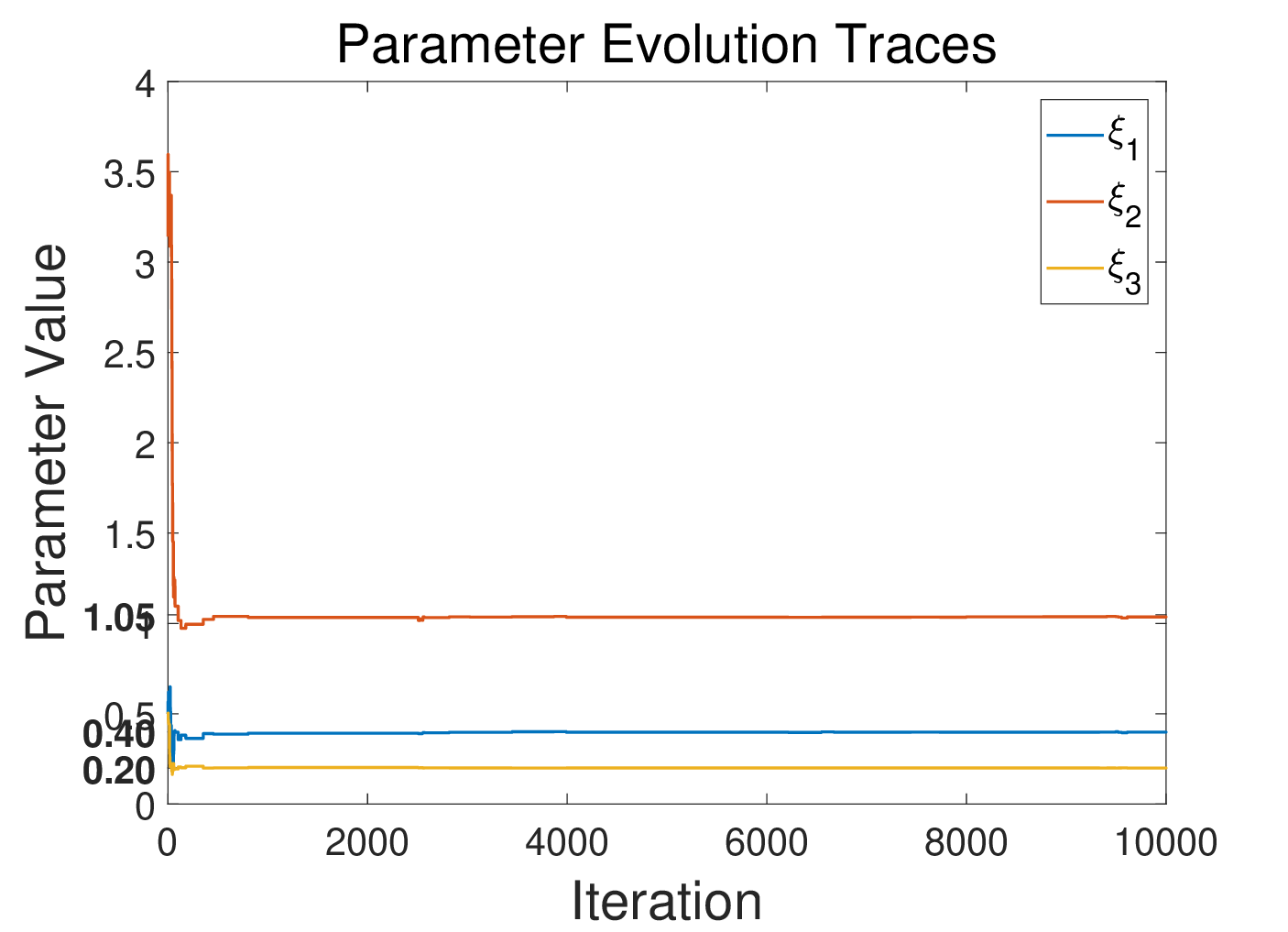}}
\caption{Trace plots of the posterior samples for the kite-shaped source with $N=10000$. The true parameters are $\boldsymbol{\xi}_{\text{true}}=(0.4,\pi/3,0.2)^\top$. The moving sensor strategy (right figure) results in narrower confidence bands, indicating decreased sample variance.}
\label{fig:bayekite}
\end{figure}

To evaluate the efficiency of the method under a smaller computational budget, we also consider the case $N=500$. Figure~\ref{fig:bayekiteacn} shows that the moving sensor strategy still achieves a visibly better approximation, confirming its practical advantage in reducing the number of required samples.

\begin{figure}[!htbp]
\centering
\subfigure[{Corresponding to Fig.~\ref{kite}(a)}]{\includegraphics[width=0.38\textwidth]{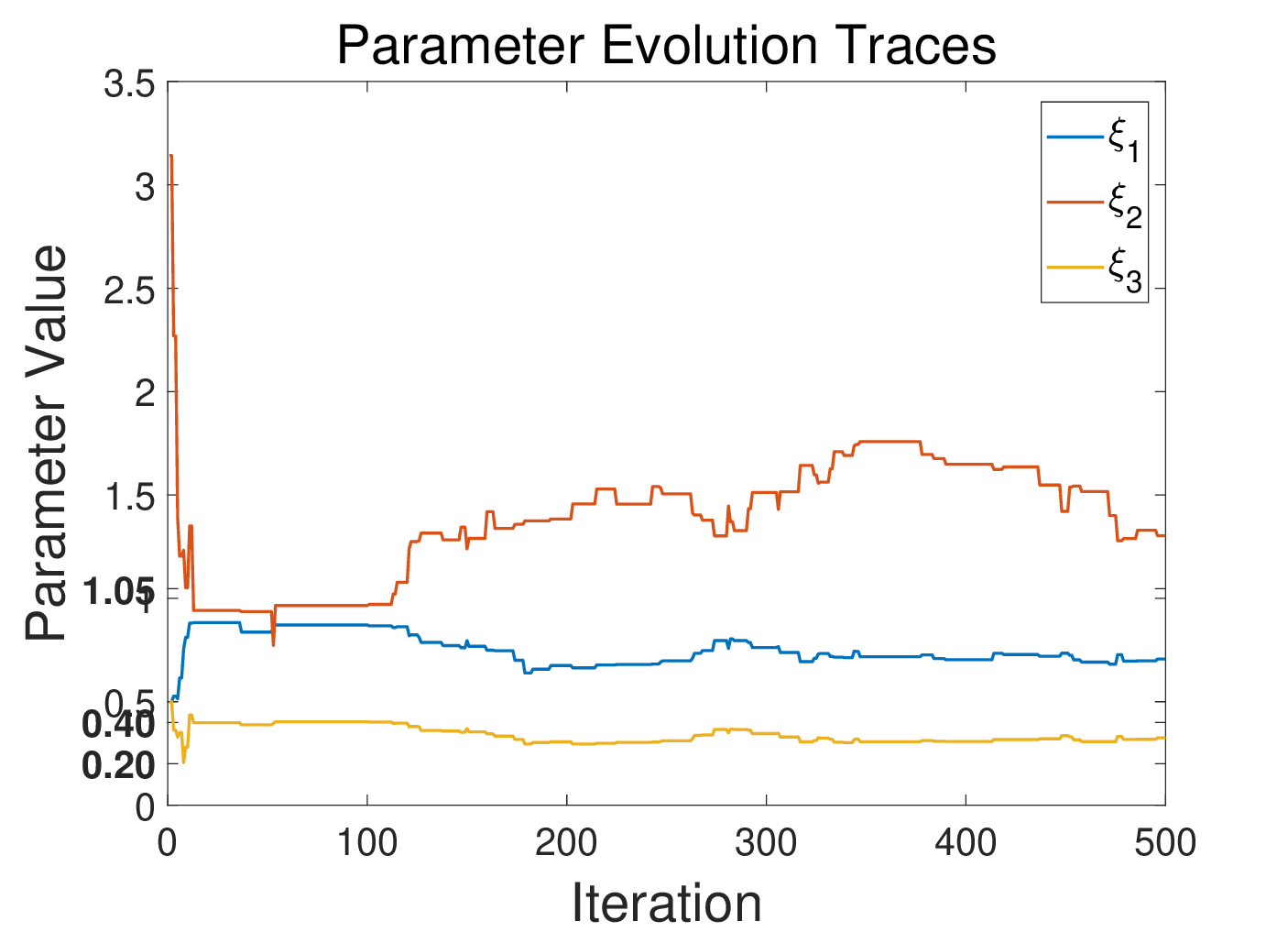}}
\subfigure[{Corresponding to Fig.~\ref{kite}(d)}]{\includegraphics[width=0.38\textwidth]{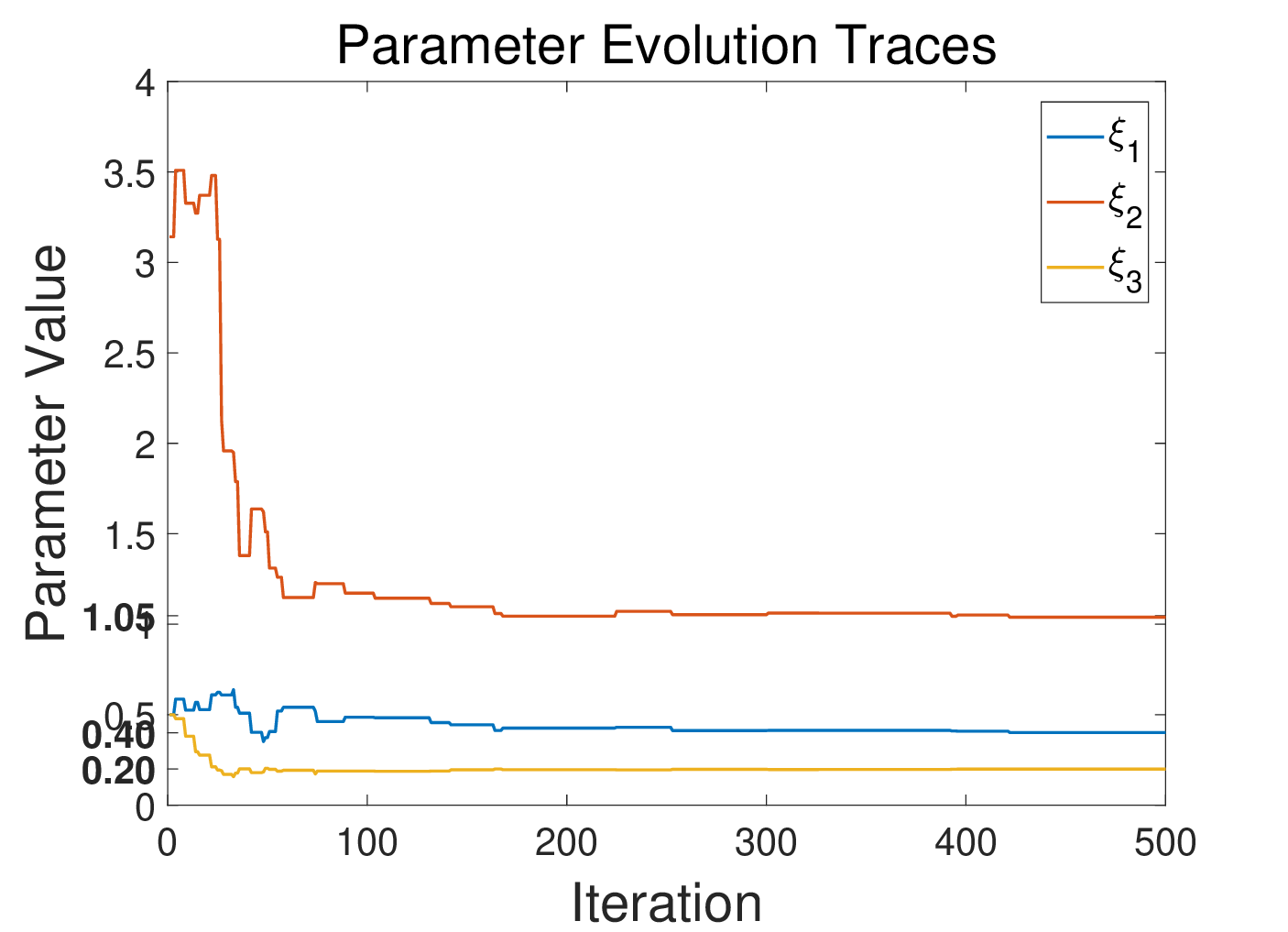}}
\caption{Trace plots of posterior samples for the kite-shaped source. The true parameters are $\boldsymbol{\xi}_{\text{true}}=(0.4,\pi/3,0.2)^\top$. With only $N=500$ iterations, the moving sensor placement strategy (right figure) achieves a more accurate approximation compared to the left figure.}
\label{fig:bayekiteacn}
\end{figure}

\FloatBarrier

\begin{example}[Four-Leaf Shaped Source]
A four-leaf shaped source is defined by
\begin{align*}
r(\theta) &= \xi_3\bigl[1+0.2\cos(4\theta)\bigr],\\
x_1 &= \xi_1\cos\xi_2+r(\theta)\cos\theta,\\
x_2 &= \xi_1\sin\xi_2+r(\theta)\sin\theta,\qquad \theta\in[0,2\pi],
\end{align*}
with parameters $\boldsymbol{\xi}=(\xi_1,\xi_2,\xi_3)$. The same bounded transformations as in the previous examples are employed, and a Gaussian prior $\mathcal{N}(\boldsymbol{0},I_3)$ is imposed on the transformed variables.
\end{example}

The true source is centred at polar coordinates $(0.4,\pi/2)$ with size parameter $0.7$, that is,
\[
\boldsymbol{\xi}_{\text{true}}=(0.4,\pi/2,0.7)^\top.
\]
We set $b=50$, $\sigma^2=0.05^2$, and use the adaptive pCN-MCMC with $N_1=0$, $N=10^4$, $k_0=2.5\times10^3$, and step sizes tuned for an acceptance rate between $25\%$ and $35\%$. In Algorithm~\ref{alg:adaptive_placement}, we set $m=10$, $c_1=1/20$, $c=20\pi$, and $N_t=80$.

Figure~\ref{leaf} shows the reconstruction process for the four-leaf shaped source. The posterior mean evolves steadily toward the true multi-lobed geometry, indicating that the proposed method is capable of handling more intricate source shapes while maintaining stable reconstruction quality from limited boundary data.

\begin{figure}[!htbp]
\centering
\subfigure[{$T\in[1/400,80/400]$, sensor at $29\pi/40$.}]{\includegraphics[width=0.38\textwidth]{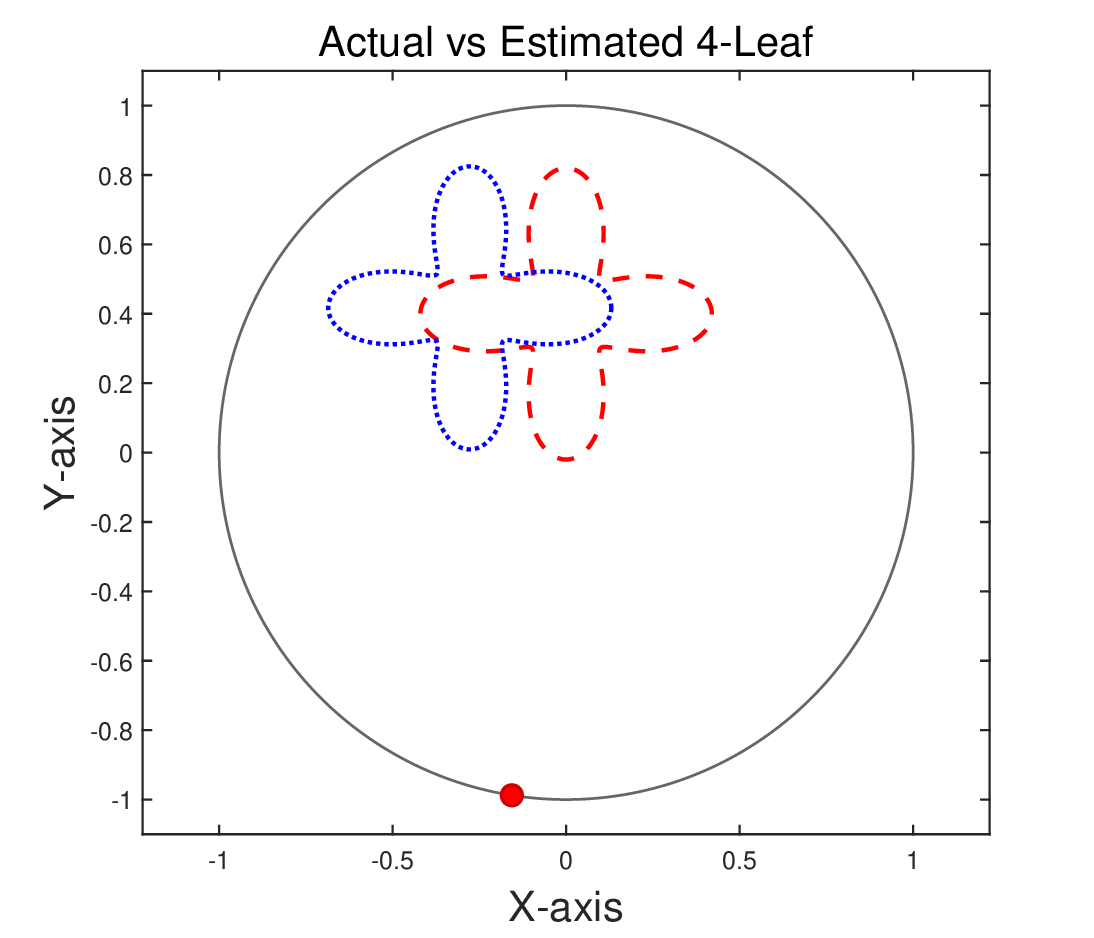}}
\subfigure[{$T\in[90/400,170/400]$, sensor at $19\pi/40$.}]{\includegraphics[width=0.38\textwidth]{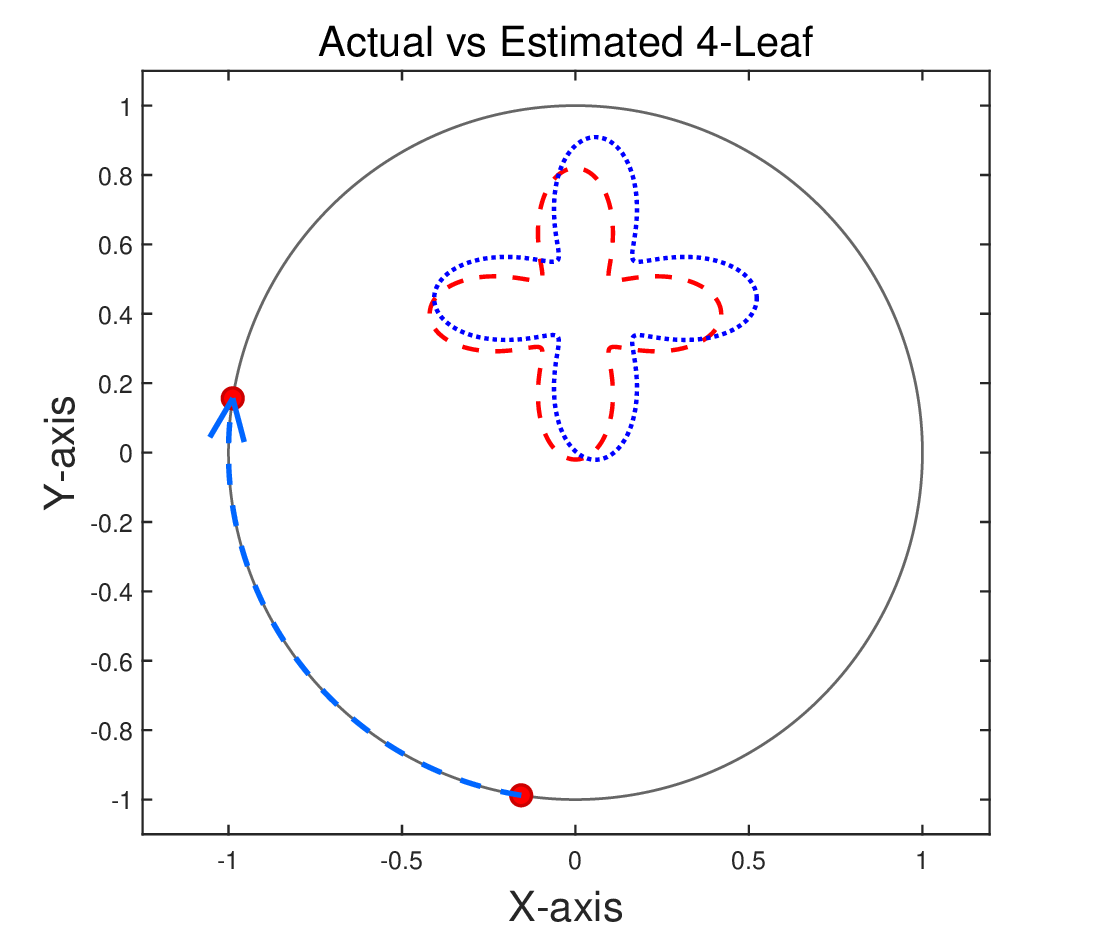}}
\subfigure[{$T\in[180/400,260/400]$, sensor at $9\pi/40$.}]{\includegraphics[width=0.38\textwidth]{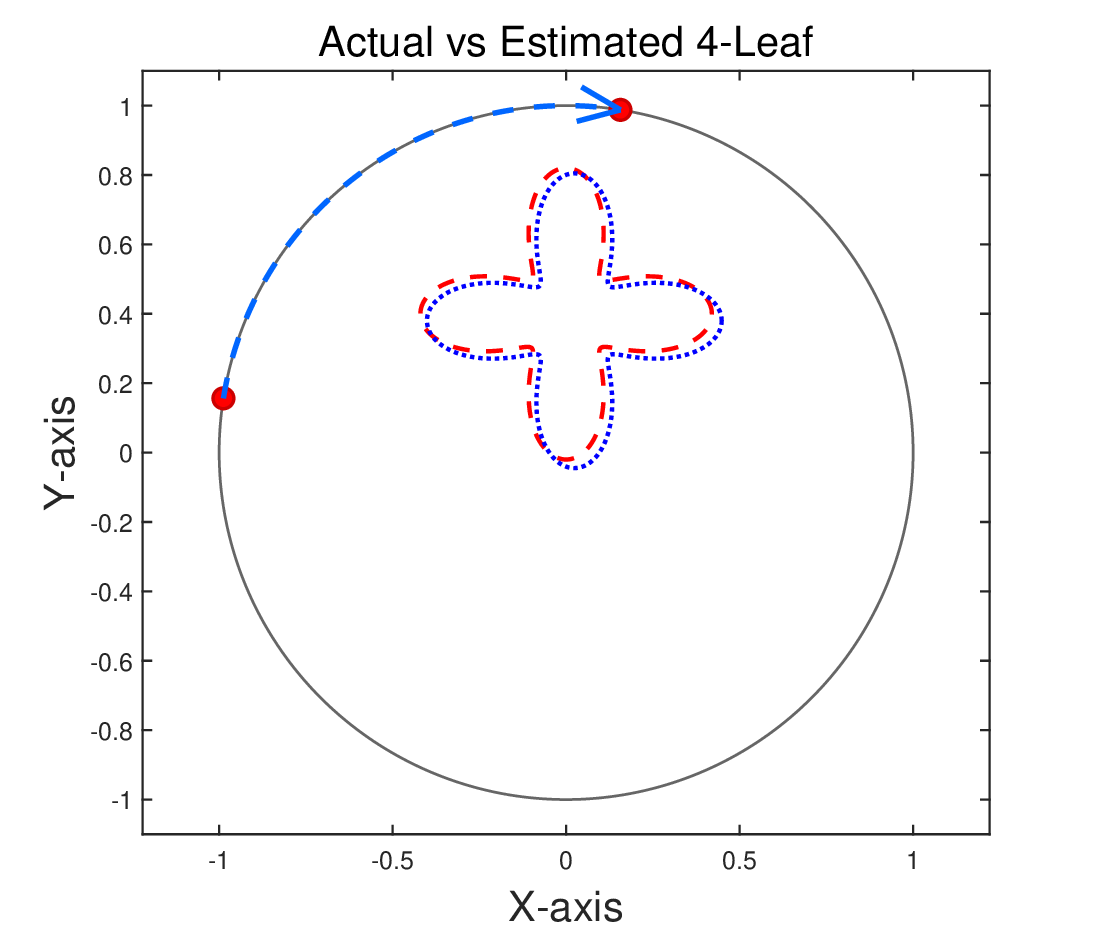}}
\subfigure[{$T\in[265/400,345/400]$, sensor at $14\pi/40$.}]{\includegraphics[width=0.38\textwidth]{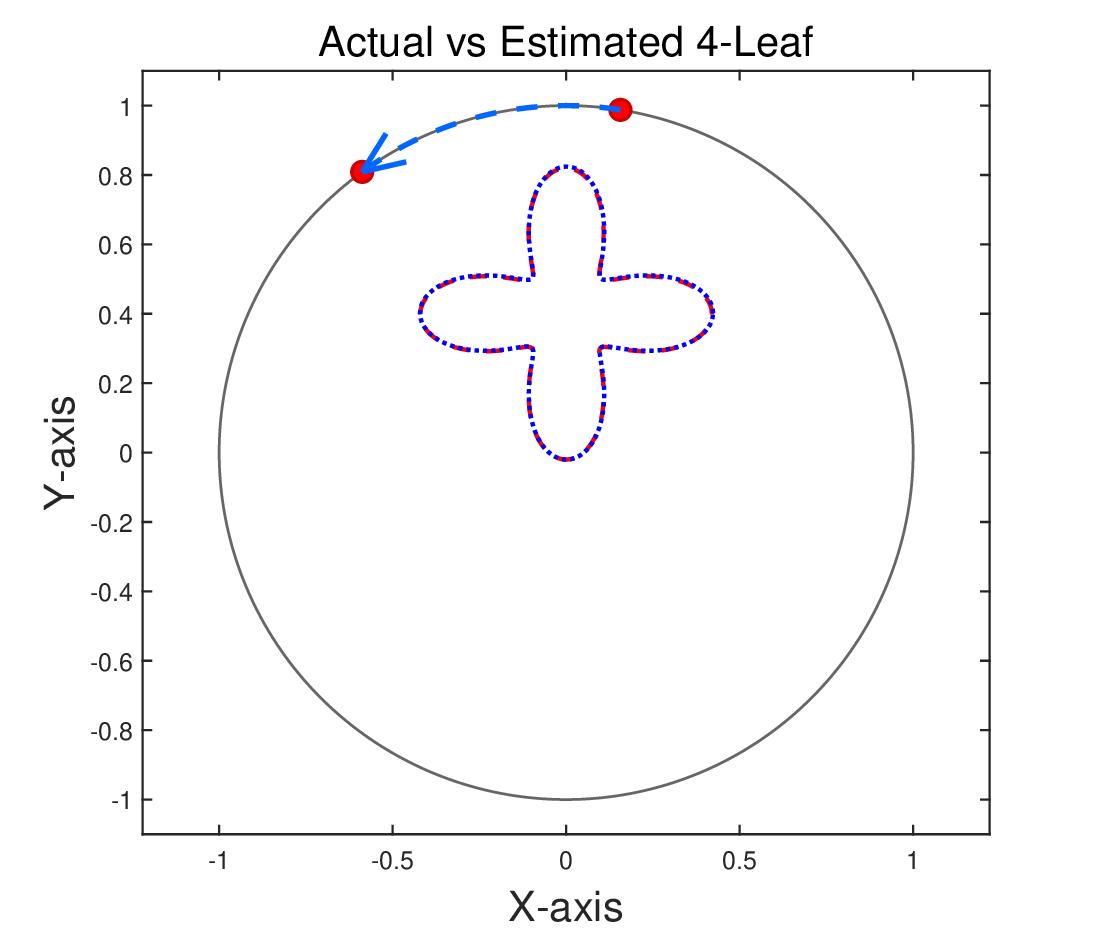}}
\caption{Evolution of the reconstruction for a four-leaf shaped source with $N=10000$.}
\label{leaf}
\end{figure}

To quantify the improvement brought by the moving sensor placement, Figure~\ref{fig:bayeleaf} reports the trace plots of the posterior samples at the initial and final stages. A consistent reduction in variance is observed in the moving-sensor stage, confirming that sequential data acquisition effectively decreases the uncertainty associated with the inferred parameters.

\begin{figure}[!htbp]
\centering
\subfigure[{Corresponding to Fig.~\ref{leaf}(a)}]{\includegraphics[width=0.38\textwidth]{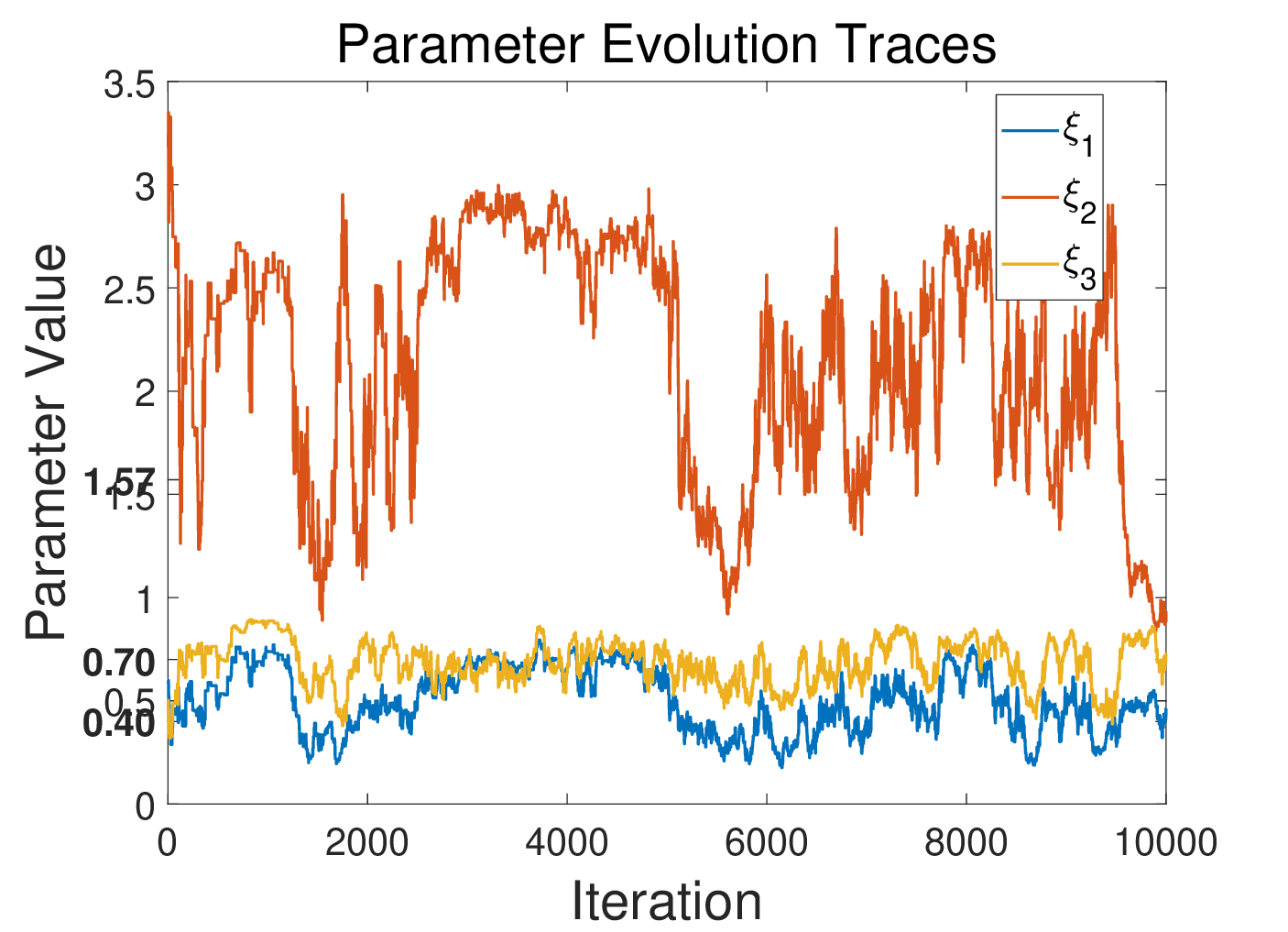}}
\subfigure[{Corresponding to Fig.~\ref{leaf}(d)}]{\includegraphics[width=0.38\textwidth]{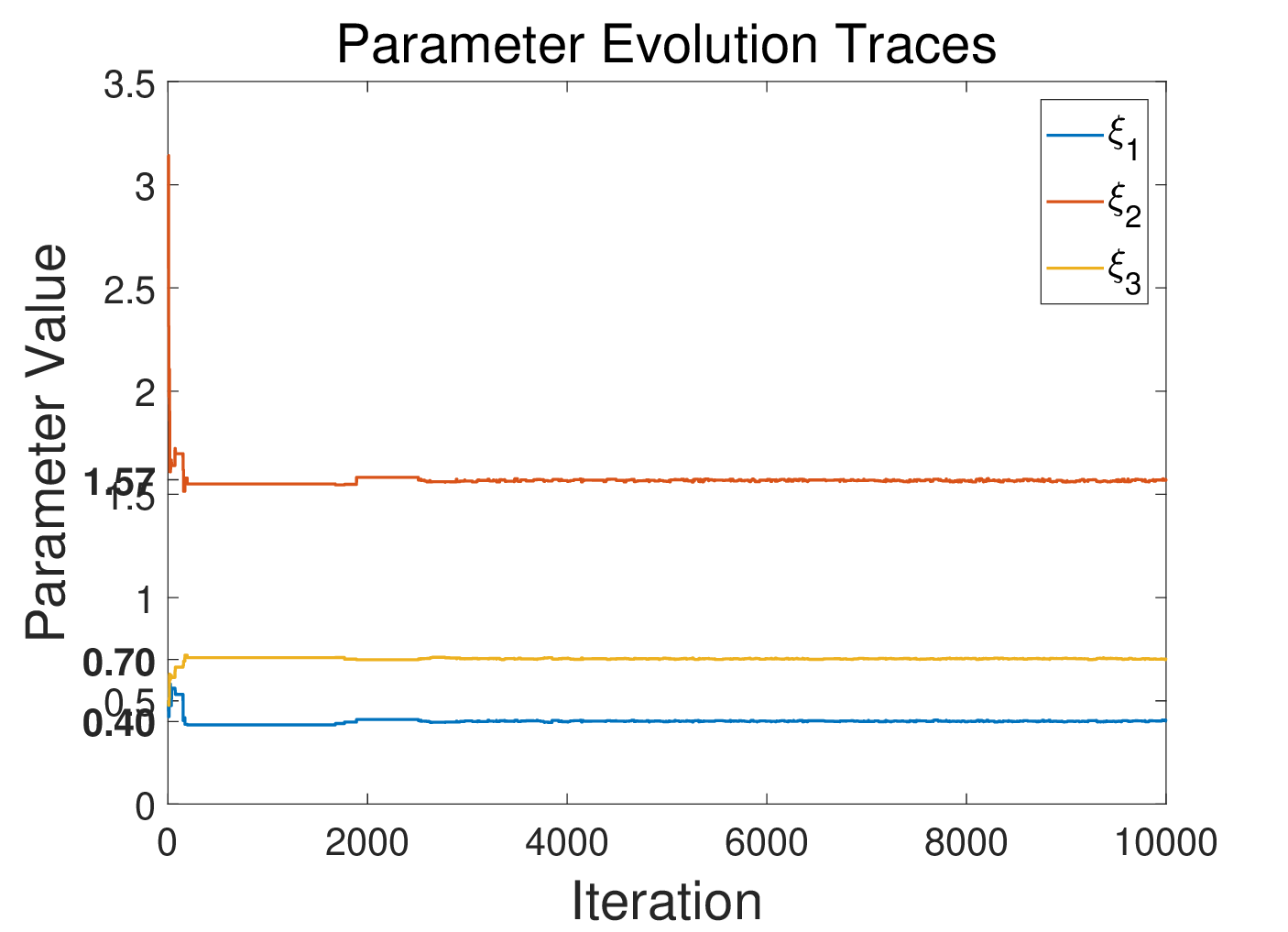}}
\caption{Trace plots of the posterior samples for the four-leaf shaped source with $N=10000$. The true parameters are $\boldsymbol{\xi}_{\text{true}}=(0.4,\pi/2,0.7)^\top$. The moving strategy (right figure) leads to a narrower confidence band, signifying decreased variance.}
\label{fig:bayeleaf}
\end{figure}

We also test the method with a reduced number of MCMC iterations, namely $N=500$. As illustrated in Figure~\ref{fig:bayeleafacn}, the moving sensor strategy continues to provide a more accurate approximation of the true source than the fixed-sensor strategy, further supporting the efficiency of the proposed framework.

\begin{figure}[!htbp]
\centering
\subfigure[{Corresponding to Fig.~\ref{leaf}(a)}]{\includegraphics[width=0.38\textwidth]{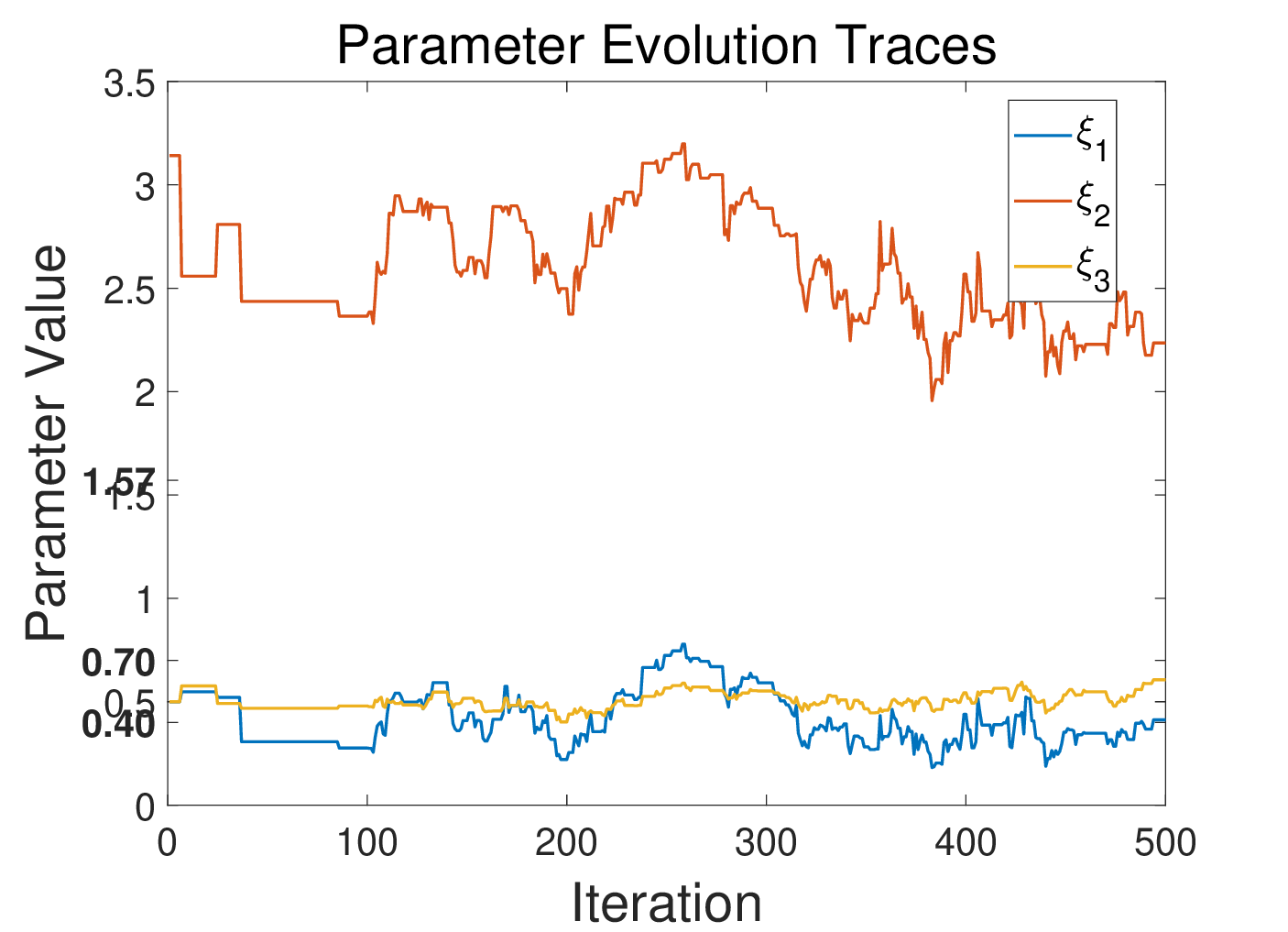}}
\subfigure[{Corresponding to Fig.~\ref{leaf}(d)}]{\includegraphics[width=0.38\textwidth]{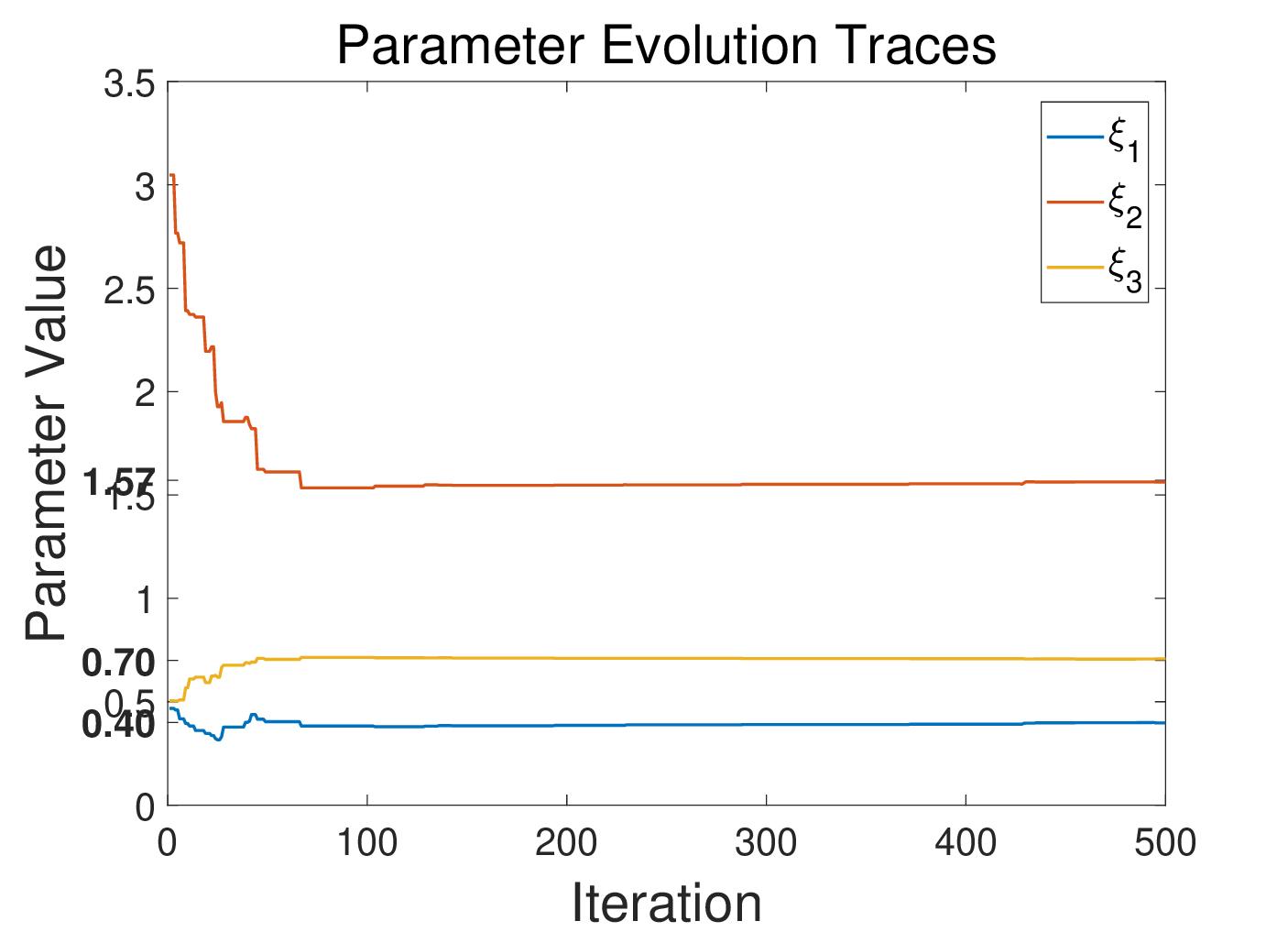}}
\caption{Trace plots of posterior samples for the four-leaf shaped source. The true parameters are $\boldsymbol{\xi}_{\text{true}}=(0.4,\pi/2,0.7)^\top$. With only $N=500$ iterations, the moving sensor placement strategy (right figure) achieves a more accurate approximation compared to the left figure.}
\label{fig:bayeleafacn}
\end{figure}

\FloatBarrier

\begin{example}[High-Dimensional Peanut-Shaped Source]
This example considers an asymmetric peanut-shaped source in a higher-dimensional parameter space. The boundary of $D$ is described in polar coordinates by
\begin{equation}\label{eq:param_boundary}
\partial D=\bigl\{\, q(\theta;\boldsymbol{\xi})\,(\cos\theta,\sin\theta)^{\!\mathsf T}:\theta\in[0,2\pi)\,\bigr\},
\end{equation}
where $q(\theta;\boldsymbol{\xi})\in(0,1)$ is a smooth $2\pi$-periodic function of $\theta$, and $\boldsymbol{\xi}\in\mathbb{R}^{2M+1}$ is the vector of unknown shape parameters. The function $q$ is expanded as the finite Fourier series
\begin{equation}\label{eq:q_series}
q(\theta;\boldsymbol{\xi})=\frac12\,\xi_1+\sum_{i=1}^{M}\Bigl(\xi_{2i}\cos(i\theta)+\xi_{2i+1}\sin(i\theta)\Bigr).
\end{equation}
\end{example}

To promote smooth reconstructions, we penalise the $H^2$ norm of $q(\theta;\boldsymbol{\xi})$, which is equivalent to imposing a zero-mean Gaussian prior $\mathcal{N}(\boldsymbol{0},B)$ on $\boldsymbol{\xi}$ \cite{RundellZhang:2018,LinOuZhangZhang:2024}. The prior covariance matrix $B\in\mathbb{R}^{(2M+1)\times(2M+1)}$ is diagonal with entries
\begin{equation}\label{eq:prior_cov}
B_{1,1}=1,\qquad
B_{2i,2i}=B_{2i+1,2i+1}=\frac{1}{i^2},\quad i=1,\dots,M.
\end{equation}
The decay factor $1/i^2$ suppresses high-frequency oscillations and therefore favours smooth boundaries.

In the simulation, we take $M=2$ and choose the true shape parameters as
\[
\boldsymbol{\xi}_{\text{true}}=[1,\;0,\;0,\;0,\;0.3]^{\mathsf T}.
\]
The source strength is set to $b=10$, and the measurement-noise variance is $\sigma^2=0.01^2$. The MCMC sampling uses $N_1=1.0\times10^3$ burn-in iterations and $N=1.5\times10^4$ total iterations, with the empirical covariance updated every $k_0=2.5\times10^3$ iterations. The step-size parameters $\beta_1$ and $\beta_2$ are tuned independently so that the acceptance rate remains between $25\%$ and $35\%$. For Algorithm~\ref{alg:adaptive_placement}, we set $m=15$, $c_1=1/20$, $c=30\pi$, and $N_t=80$.

Figure~\ref{Peanut} presents the reconstruction results over four time windows for the high-dimensional peanut-shaped source. Despite the increased parameter dimension and the more complex geometric structure of the target, the proposed method still exhibits stable convergence and produces satisfactory reconstructions from sparse moving observations.

\begin{figure}[!htbp]
\centering
\subfigure[{$T\in[1/400,80/400]$, sensor at $4\pi/40$.}]{\includegraphics[width=0.38\textwidth]{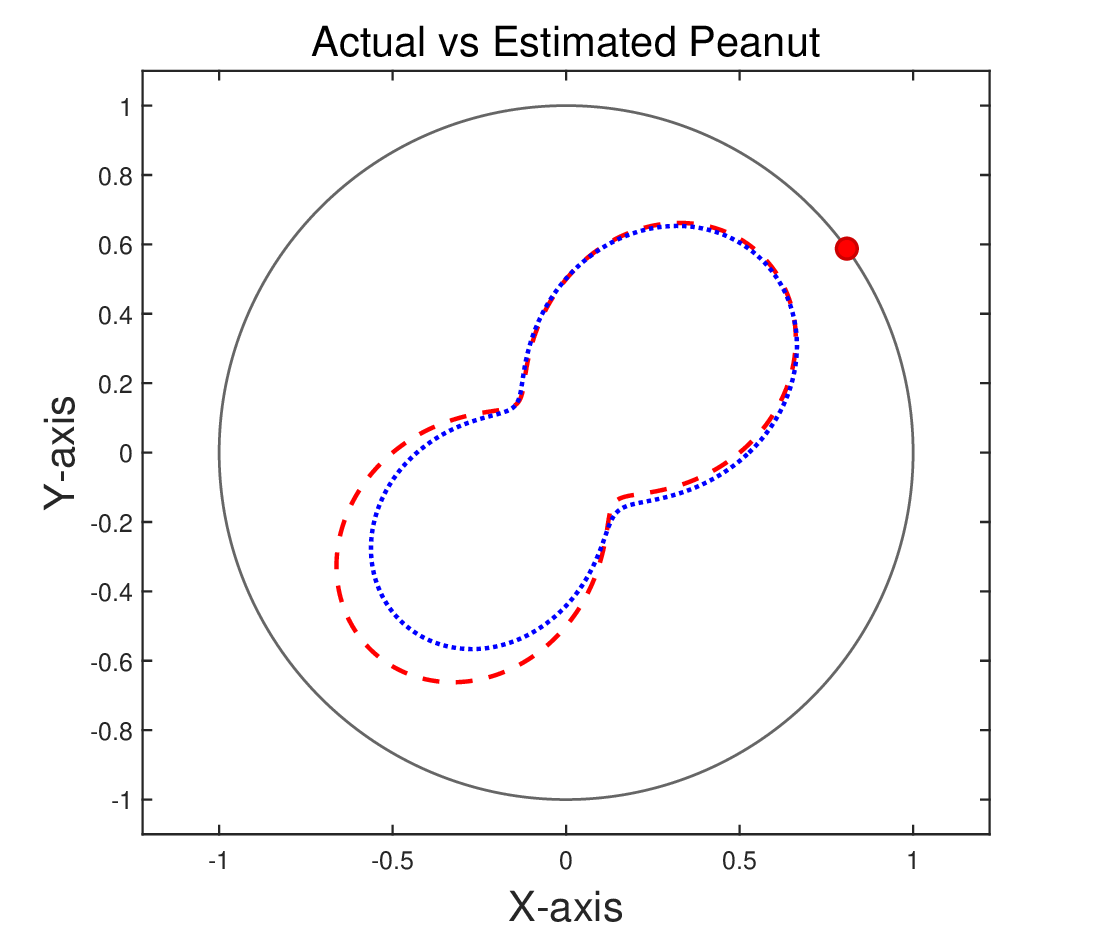}}
\subfigure[{$T\in[90/400,170/400]$, sensor at $19\pi/40$.}]{\includegraphics[width=0.38\textwidth]{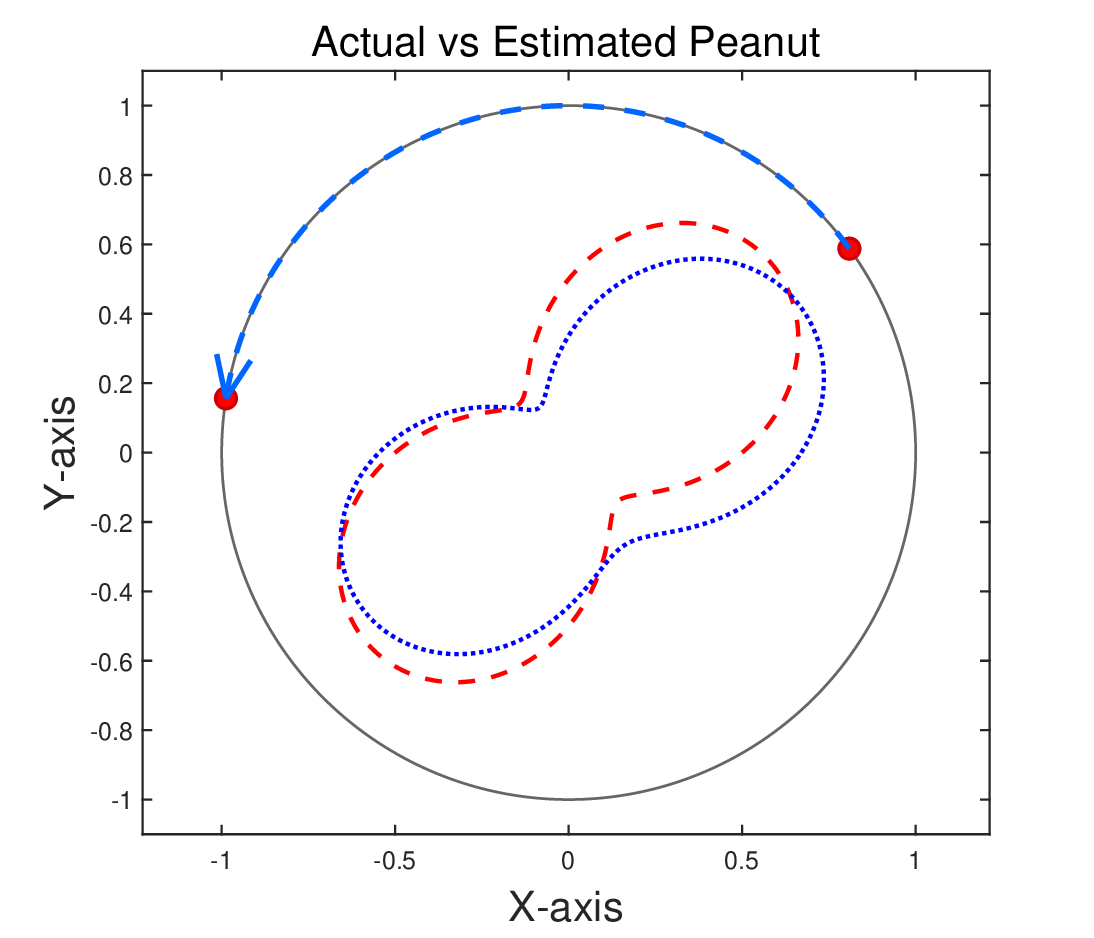}}
\subfigure[{$T\in[180/400,260/400]$, sensor at $34\pi/40$.}]{\includegraphics[width=0.38\textwidth]{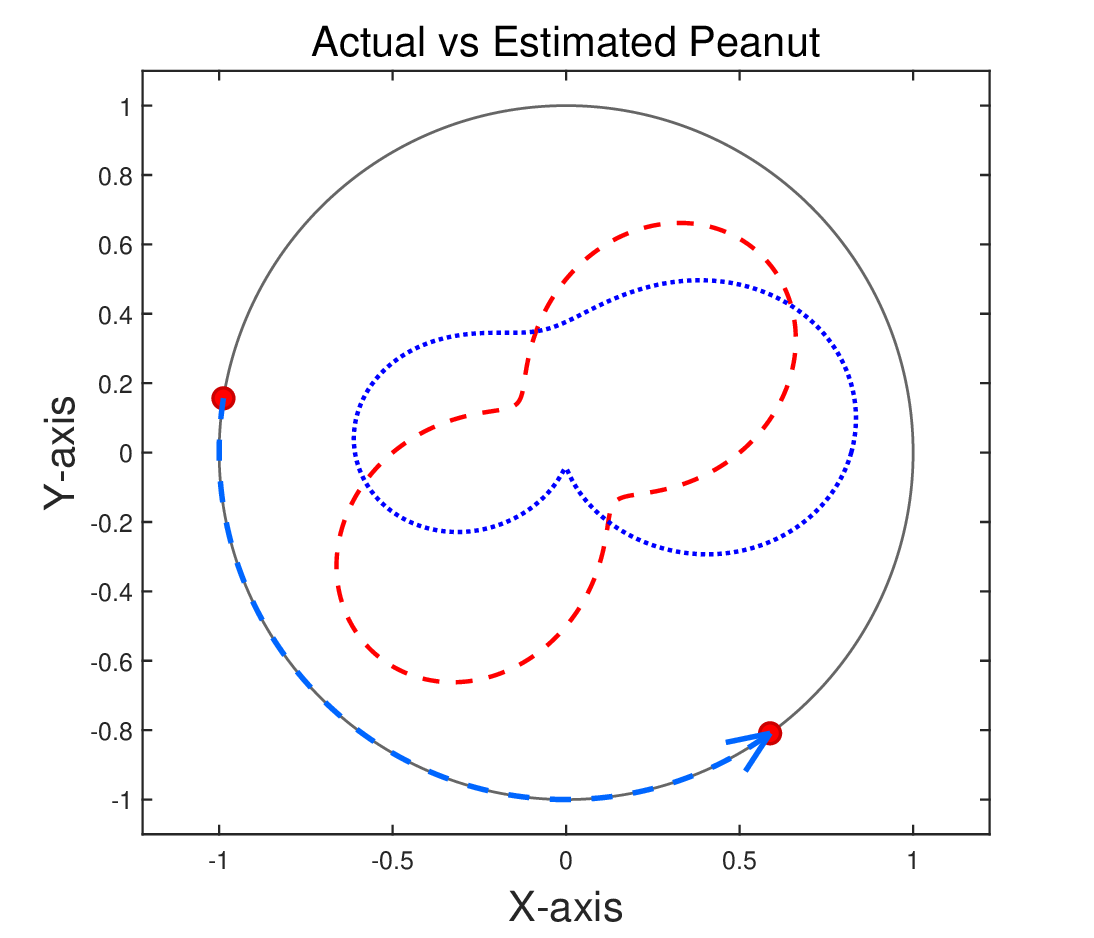}}
\subfigure[{$T\in[794/1200,1034/1200]$, sensor at $27\pi/40$.}]{\includegraphics[width=0.38\textwidth]{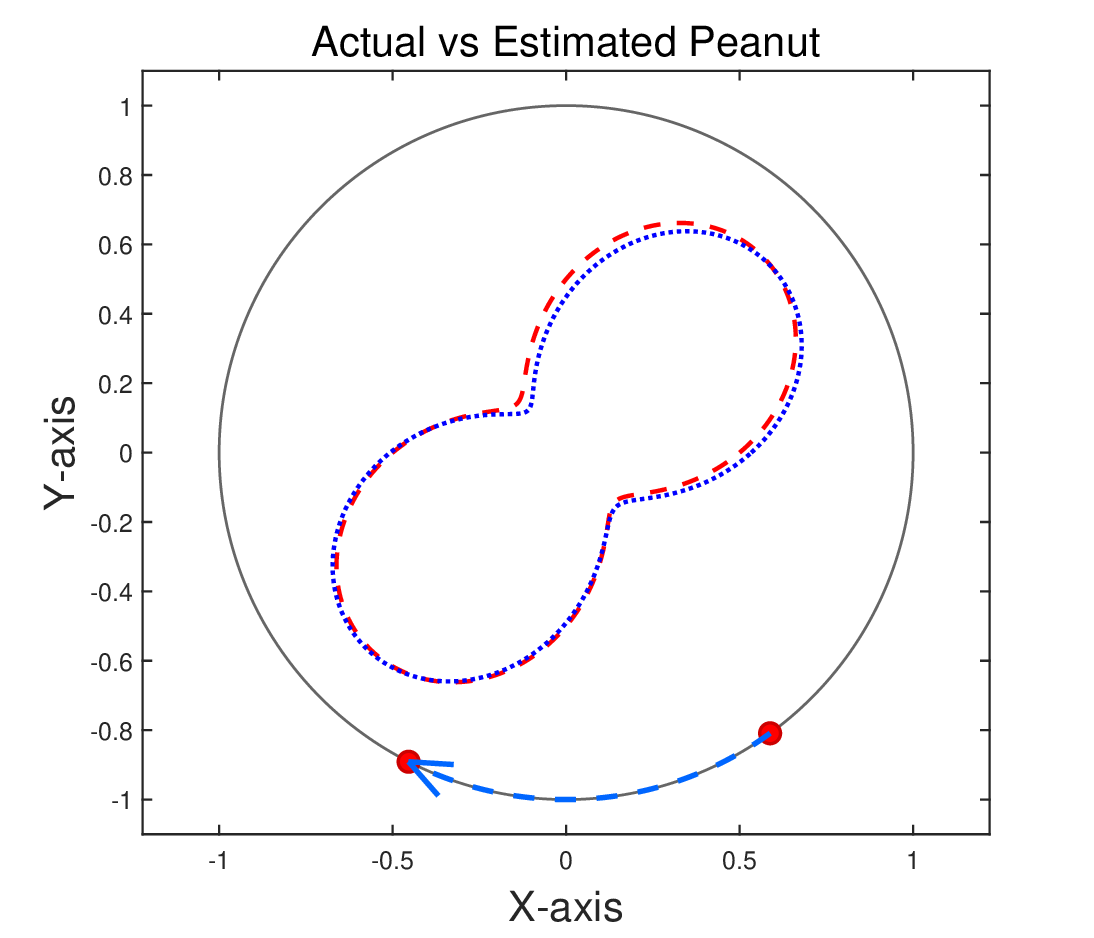}}
\caption{Evolution of the reconstruction for a peanut-shaped source with $N=15000$.}
\label{Peanut}
\end{figure}

The posterior trace plots for the five parameters are displayed in Figure~\ref{fig:bayepeanut}. In comparison with the fixed-sensor stage, the moving-sensor stage yields substantially tighter sample distributions. In particular, $\xi_1$ and $\xi_5$ are seen to concentrate around their true values, while $\xi_2$, $\xi_3$, and $\xi_4$ remain centred near zero. These results clearly illustrate the uncertainty reduction achieved by the adaptive sensor movement.

\begin{figure}[!htbp]
\centering
\subfigure[{Corresponding to Fig.~\ref{Peanut}(a)}]{\includegraphics[width=0.38\textwidth]{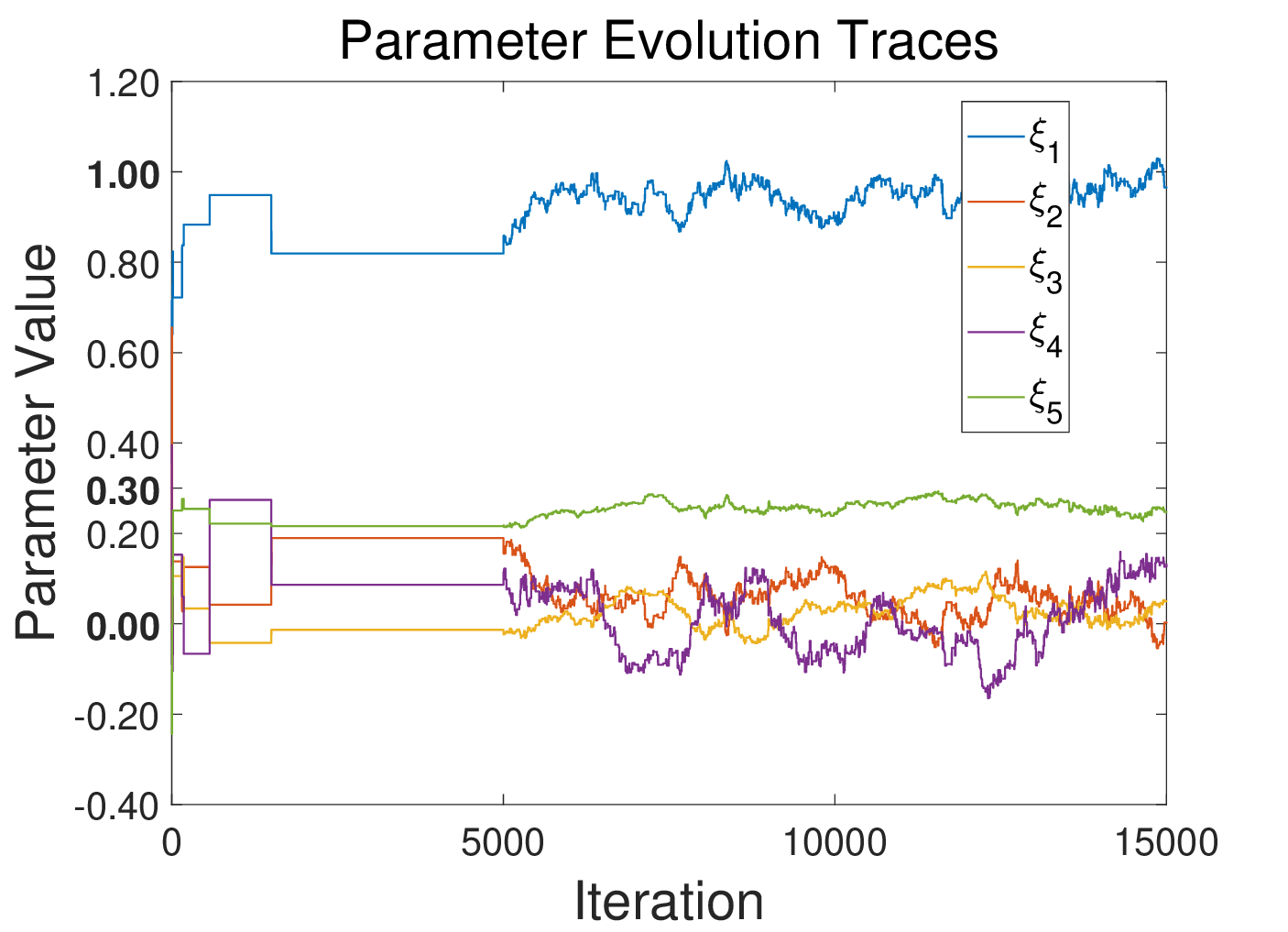}}
\subfigure[{Corresponding to Fig.~\ref{Peanut}(d)}]{\includegraphics[width=0.38\textwidth]{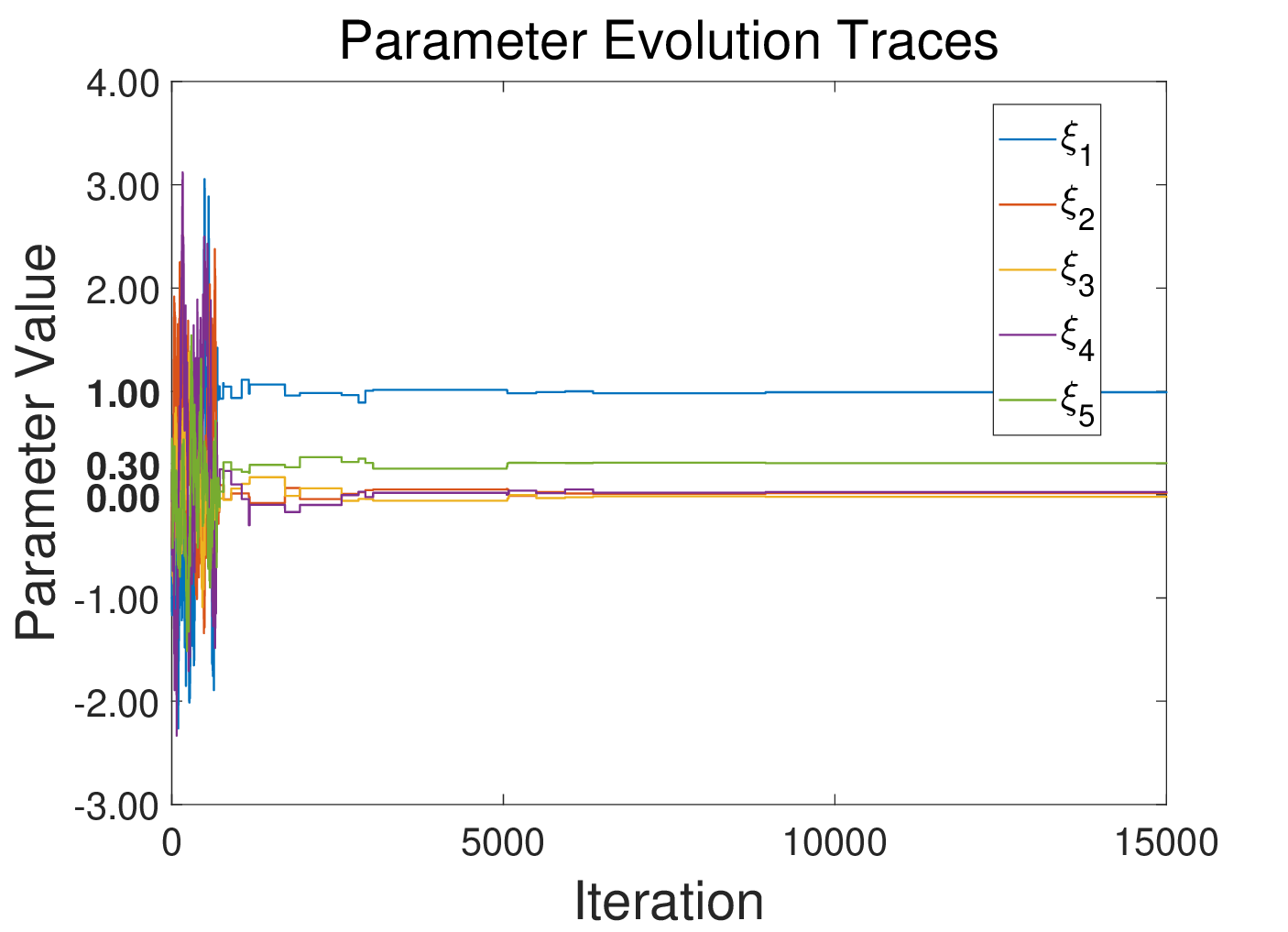}}
\caption{Trace plots of the posterior samples for the peanut-shaped source with $N=15000$. The true parameters are $\boldsymbol{\xi}_{\text{true}}=(1,0,0,0,0.3)^\top$. The moving sensor placement strategy (right figure) leads to samples where $\xi_1$ and $\xi_5$ concentrate near $1$ and $0.3$, and $\xi_2$, $\xi_3$, $\xi_4$ near $0$, with markedly narrower confidence bands indicating reduced variance.}
\label{fig:bayepeanut}
\end{figure}

Finally, to further demonstrate the efficiency of the proposed strategy in the higher-dimensional setting, we also report results obtained with a reduced total of $N=3500$ iterations. Figure~\ref{fig:bayepeanutacn} shows that, even under this reduced sampling budget, the moving sensor strategy continues to provide a significantly better approximation to the true source than the fixed-sensor strategy.

\begin{figure}[!htbp]
\centering
\subfigure[{Corresponding to Fig.~\ref{Peanut}(a)}]{\includegraphics[width=0.38\textwidth]{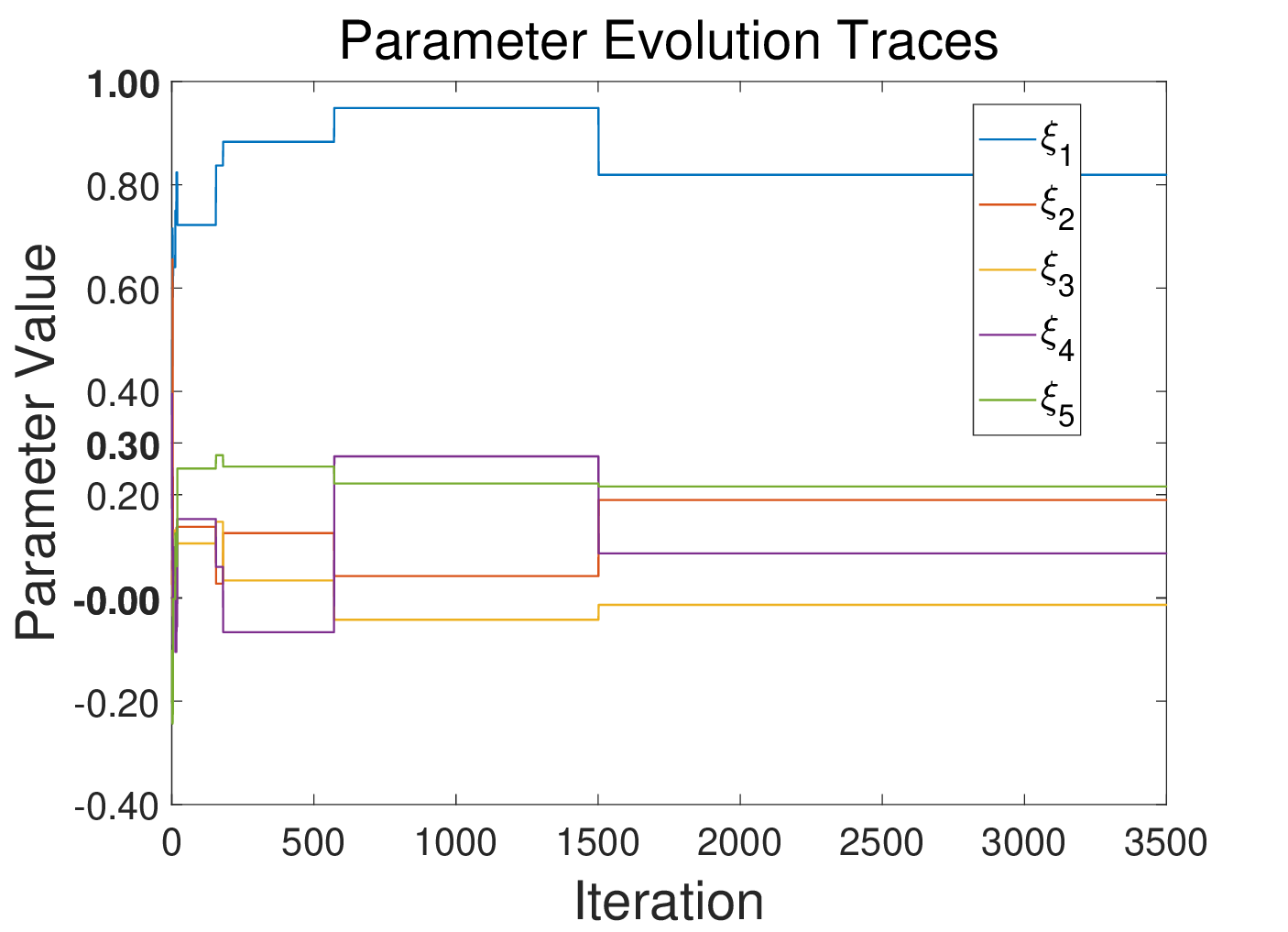}}
\subfigure[{Corresponding to Fig.~\ref{Peanut}(d)}]{\includegraphics[width=0.38\textwidth]{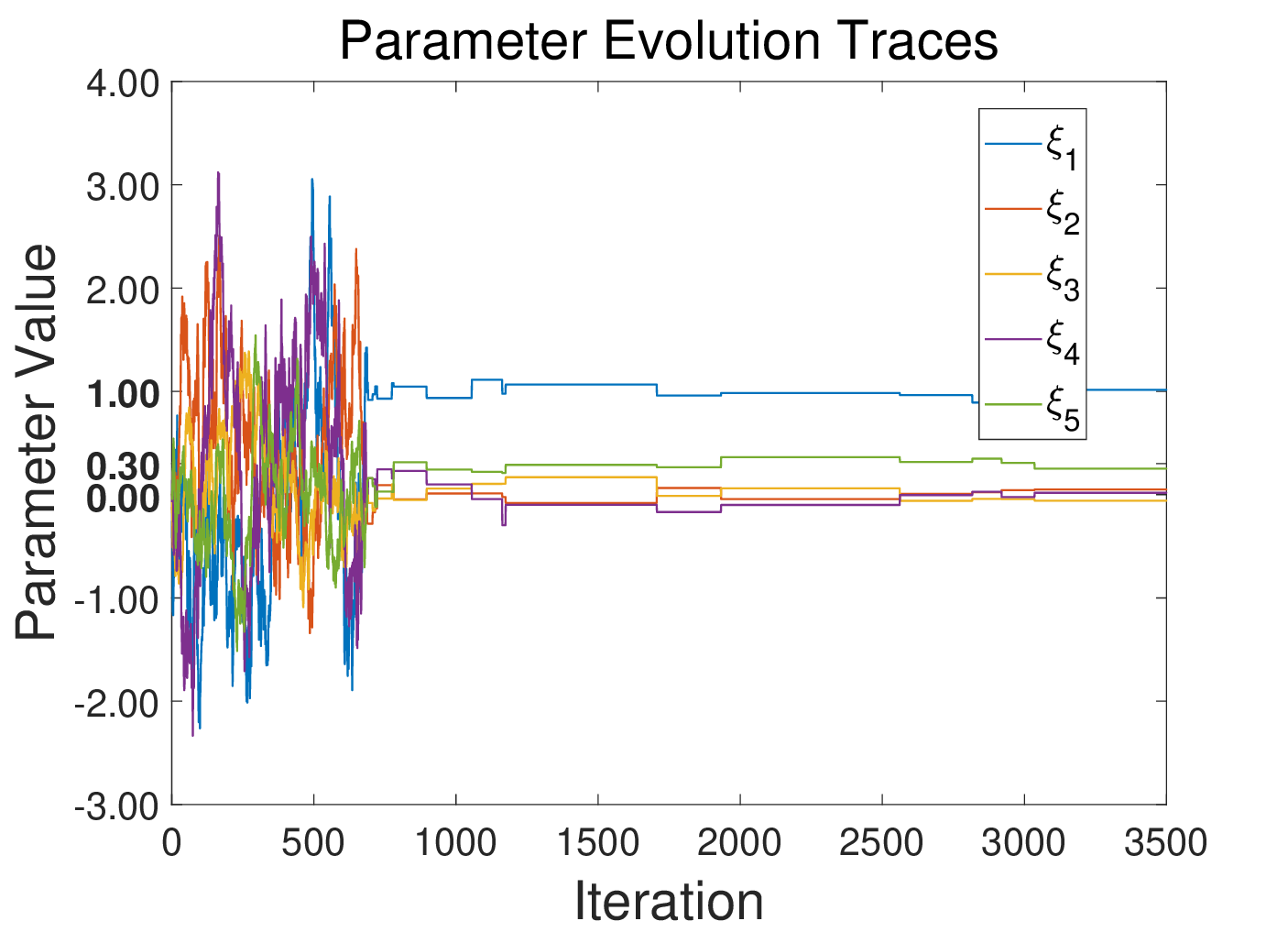}}
\caption{Trace plots of posterior samples for the peanut-shaped source. The true parameters are $\boldsymbol{\xi}_{\text{true}}=(1,0,0,0,0.3)^\top$. With only $N=3500$ iterations, the moving sensor placement strategy (right figure) achieves a significantly better approximation to the true solution compared to the left figure.}
\label{fig:bayepeanutacn}
\end{figure}

\FloatBarrier

\section*{Acknowledgments.} 
Zhidong Zhang is supported by the National Key Research and Development Plan of China(Grant No.2023YFB3002400). Wenlong Zhang is partially supported by the National Natural Science Foundation of China under grant numbers No.12371423, No.12241104 and No.12561160122.

\bibliographystyle{abbrv}
\bibliography{ref}

\end{document}